\documentclass[onefignum,onetabnum,reqno]{siamart251216}

\usepackage{graphicx} % Required for inserting images
\usepackage{ragged2e}
\usepackage{blindtext}
\usepackage[utf8]{inputenc}
\usepackage[T1]{fontenc}
\usepackage{amsmath, amssymb, mathtools, bm, esint}
\usepackage{booktabs}
\usepackage{mathrsfs}
\usepackage{enumerate}
\usepackage{comment}
\usepackage{indentfirst}
\usepackage{empheq}
\usepackage{xcolor,colortbl}

\usepackage{cleveref}
\usepackage{algorithm}
\usepackage{algorithmic}

\usepackage{paralist} 

\usepackage{adjustbox}

\newtheorem{Theo}{Theorem}[section]
\newtheorem{Rem}{Remark}[section]
\newtheorem{Def}{Definition}[section]
\newtheorem{Lem}{Lemma}[section]
\newtheorem{Prop}{Proposition}[section]

\newtheorem{Ex}{Example}[section]
\newtheorem{Assum}{Assumption}[section]

\numberwithin{equation}{section}
\numberwithin{figure}{section}
\numberwithin{table}{section}
\numberwithin{algorithm}{section}

\renewcommand{\emptyset}{\varnothing}

\providecommand{\bx}{\boldsymbol{x}}
\providecommand{\bxi}{\boldsymbol{\xi}}
\providecommand{\bzeta}{\boldsymbol{\zeta}}
\newcommand{\dx}{\,\mathrm{d}\bx}
\newcommand{\ds}{\,\mathrm{d}s}

\providecommand{\U}{\mathbb{U}}
\providecommand{\V}{\mathbb{V}}
\providecommand{\X}{\mathbb{X}}

\providecommand{\M}{\mathbb{M}}

\providecommand{\T}{\mathcal{T}}
\providecommand{\E}{\mathcal{E}}
\renewcommand{\P}{\mathbb{P}}
\newcommand{\I}{\mathcal{I}}
\newcommand{\C}{\mathscr{C}}
\newcommand{\F}{\mathcal{F}}
\newcommand{\osc}{\mathrm{osc}}

\newcommand{\DProd}[1]{\left\langle#1\right\rangle}
\newcommand{\norm}[1]{\left\| #1 \right\|}
\usepackage{setspace}
\allowdisplaybreaks

\headers{MinRes for nonlinear PDEs in Banach spaces}{I. Muga, J. Perera, S. Rojas, and R. Ruiz-Baier}

\title{A Residual Minimization approach for Nonlinear Partial Differential Equations set in Banach spaces\thanks{\textbf{Updated:} \today. \funding{This work has been partially supported by ANID through FONDECYT projects 1240643 (SR) and 1230091 (IM), by the ANID National Doctoral Scholarship 21250504 (JP), by the National Center for Artificial Intelligence CENIA FB210017, Basal ANID, by the Australian Research Council through the \textsc{Future Fellowship} Grant FT220100496, and by the Center of Advanced Study (CAS) at the Norwegian Academy of Science and Letters under the program \textsc{Mathematical Challenges in Brain Mechanics}.}}}

\author{Ignacio Muga\thanks{Instituto de Matem\'aticas, Pontificia Universidad Cat\'olica de Valpara\'{\i}so, Valpara\'iso, Chile (\email{ignacio.muga@pucv.cl}).} 
\and 
Jorge Perera\thanks{Corresponding author: Instituto de Matem\'aticas, Pontificia Universidad Cat\'olica de Valpara\'{\i}so, Valpara\'iso, Chile (\email{jorgeperera1597@gmail.com}).} 
\and Sergio Rojas\thanks{School of Mathematics, Monash University, 9 Rainforest Walk, 3800 Melbourne VIC, Australia (\email{sergio.rojas@monash.edu}).} 
\and Ricardo Ruiz-Baier\thanks{School of Mathematics, Monash University, 9 Rainforest Walk, Melbourne VIC 3800, Australia, and Universidad Adventista de Chile, Casilla 7-D Chill\'an, Chile (\email{ricardo.ruizbaier@monash.edu}).}}

\date{\today}

\begin{document}

\maketitle

\begin{abstract}
In this work, we propose and analyze a residual-minimization strategy for the numerical solution of nonlinear PDEs posed in Banach spaces. Given a finite-dimensional trial space and a suitably enriched discrete test space (of higher dimension than the trial space), we approximate the solution by minimizing the variational residual in a discrete dual norm. This minimization is equivalent to a nonlinear saddle-point formulation for the discrete solution in the trial space together with a residual representative in the test space. The latter provides a natural a posteriori error estimator, enabling automatic mesh adaptivity. To solve the resulting nonlinear saddle-point problem, we propose a Newton iteration whose linearized saddle-point system is symmetric, thereby guaranteeing solvability at each step. We take the $p$--Laplacian as a model problem and support the theoretical developments with representative numerical experiments,  using standard $H^1$--conforming piecewise linear functions for the trial space, and lowest-order Crouzeix--Raviart functions for the test space.
\end{abstract}
\begin{keywords}
Residual minimization, nonlinear PDEs, Banach spaces, $p$-Laplacian, adaptivity. 
\end{keywords}
\begin{MSCcodes} 65N30, 65N12, 65N15.\end{MSCcodes}

\section{Introduction}
\subsection{Scope} 
Partial Differential Equations (PDEs) are commonly expressed in the abstract form:
\begin{equation}\label{abstractPDE}
    \begin{cases}
        \text{Find } u \in \U:\\
        A(u) = F\ \ \text{in } \V^*,
    \end{cases}
\end{equation}
where {the mapping $A: \U \to \V^*$ represents the underlying PDE,} $\U$ and $\V$ are infinite-dimensional Banach spaces and $F \in \V^*$ is a given functional.
Although standard variational formulations form the basis of finite element discretizations, residual minimization (MinRes) methods have emerged as a robust alternative framework.
The core idea of MinRes is to reformulate the PDE as an optimization problem by minimizing the norm of the residual $F - A(u)$ in a suitable dual space.

Different MinRes strategies based on finite elements have been proposed to numerically approximate the solution(s) of \eqref{abstractPDE} in specific contexts.
For instance, when the right-hand side functional $F$ is in a Lebesgue space $L^2$, the least-squares method \cite{Bochev2009Least} can be applied, which reformulates PDE problems as the minimization of a functional that measures the residual error of the equation in the $L^2$ norm.
In the context of Hilbert spaces, the Discontinuous Petrov--Galerkin (DPG) method (see \cite{demkowicz2025discontinuous}, and references therein for a recent overview) has established itself as a powerful methodology by minimizing the residual in dual norms, leading to optimal stability properties. More recently, these MinRes principles have been extended to the training of neural networks, such as in the Robust Variational Physics-Informed Neural Networks (RVPINNs) framework (see \cite{Rojas2024Robust,fuhrer2025posteriori}), which minimizes the discrete dual norm of the residual to guarantee stable and robust error estimation.

Other methods that minimize the residual in dual norms, or tackle negative and fractional Sobolev norms in the context of Hilbert spaces, can be found in \cite{Cier2021Automatically, Cier2021ANonlinear, Monsuur2024Minimal}.
{Furthermore, the DPG framework has been extended to analyze nonlinear problems in Hilbert spaces, such as quasilinear elliptic PDEs \cite{Carstensen2018Nonlinear}, strongly monotone operators \cite{Cantin2018A}, and more recently, nonlinear evolution equations through multistage time-marching schemes \cite{Munoz2023Multistage}.}
%\cred{[There are plenty of works in the DPG community for nonlinear PDEs analysed in Hilbert spaces, look for nonlinear DPG in Google Scholar]}

Nevertheless, many physical phenomena and mathematical models are naturally governed by non-quadratic energies and, consequently, are best described in Banach spaces rather than Hilbert spaces. A classical example that arises in the modeling of non-Newtonian fluids, plasticity, glaciology, and nonlinear diffusion processes is the $p$-Laplacian problem:
Given $p \in (1, \infty)$ and a source term $f \in (W_0^{1,p}(\Omega))^*$, find $u \in W_0^{1,p}(\Omega)$ such that:
\begin{equation}\label{pLaplacian_problem}
    \begin{cases}
        -\Delta_p u = f & \text{in } \Omega, \\
        u = 0 & \text{on } \partial \Omega,
    \end{cases}
\end{equation}
where $\Delta_p u := \text{div}(|\nabla u|^{p-2} \nabla u)$ and $\Omega \subset \mathbb{R}^d$ is a bounded open set with a Lipschitz boundary $\partial \Omega$.
For such problems, where the natural function spaces are Sobolev spaces $W^{1,p}(\Omega)$ with $p \neq 2$, restricting the numerical analysis to Hilbert structures is theoretically insufficient.

An approach generalizing residual minimization methods to Banach spaces was introduced in \cite{Muga2020Discretization}. It is worth emphasizing that this framework, as well as the approaches in \cite{Cier2021Automatically, Cier2021ANonlinear}, provides a posteriori error estimators that naturally drive adaptive mesh refinement.
However, while the work in \cite{Muga2020Discretization} successfully established the theory for continuous, bounded-below linear operators, extending it to nonlinear operators remains challenging.

Recent advancements have begun to address the highly nonlinear optimization hurdles that arise when minimizing residuals in Banach spaces, proposing strategies such as variable exponents \cite{Li2022An} or regularized Kačanov iterations \cite{Storn2024Solving} for linear PDEs. {Furthermore, as highlighted in \cite{Balci2023Relaxed}, the inherent structural differences of the $p$-Laplacian necessitate analyzing the degenerate ($p \geq 2$) and singular ($1 < p < 2$) regimes separately to derive accurate stability and convergence bounds.}

In this work, we aim to further bridge this gap by developing residual minimization methods for nonlinear PDEs that provide an a posteriori error estimator to automatically drive adaptivity, particularly when $\V$ and its dual space $\V^*$ are strictly convex reflexive Banach spaces.

\subsection{Outline}

The remainder of this work is organized as follows.
    In \Cref{sec:Preliminaries}, we present the necessary preliminaries, including definitions of functional spaces, properties of duality maps in strictly convex Banach spaces, and key results on best approximation.
    \Cref{sec:MinResDualNorms} is devoted to the continuous residual minimization method, establishing its well-posedness, equivalent characterizations, and a continuous a posteriori error estimator.
    In \Cref{sec:discreteMinRes}, we analyze the residual minimization method in discrete dual norms. We discuss its stability, provide a convergence analysis, and derive a discrete a posteriori error estimator.
    Finally, in \Cref{sec:NumericalImplementation}, we detail the numerical implementation for the $p$-Laplacian problem, including the finite element spaces, the iterative algorithm, and numerical experiments validating our theoretical findings.

%%%%%%%%%%%%%%%%%%%%%%%%%%%%%%%%%%%%%%%%%%%%%%%%%%%%%%%%%%%%%%%%%%%%%%%%%%%%%%%%%%%%%%
%%%%%%%%%%%%%%%%%%%%%%%%%%%%%%%%%%%%%%%%%%%%%%%%%%%%%%%%%%%%%%%%%%%%%%%%%%%%%%%%%%%%%%
%%%%%%%%%%%%%%%%%%%%%%%%%%%%%%%%%%%%%%%%%%%%%%%%%%%%%%%%%%%%%%%%%%%%%%%%%%%%%%%%%%%%%%
%%%%%%%%%%%%%%%%%%%%%%%%%%%%%%%%%%%%%%%%%%%%%%%%%%%%%%%%%%%%%%%%%%%%%%%%%%%%%%%%%%%%%%
%%%%%%%%%%%%%%%%%%%%%%%%%%%%%%%%%%%%%%%%%%%%%%%%%%%%%%%%%%%%%%%%%%%%%%%%%%%%%%%%%%%%%%
%%%%%%%%%%%%%%%%%%%%%%%%%%%%%%%%%%%%%%%%%%%%%%%%%%%%%%%%%%%%%%%%%%%%%%%%%%%%%%%%%%%%%%
%%%%%%%%%%%%%%%%%%%%%%%%%%%%%%%%%%%%%%%%%%%%%%%%%%%%%%%%%%%%%%%%%%%%%%%%%%%%%%%%%%%%%%
%%%%%%%%%%%%%%%%%%%%%%%%%%%%%%%%%%%%%%%%%%%%%%%%%%%%%%%%%%%%%%%%%%%%%%%%%%%%%%%%%%%%%%
%%%%%%%%%%%%%%%%%%%%%%%%%%%%%%%%%%%%%%%%%%%%%%%%%%%%%%%%%%%%%%%%%%%%%%%%%%%%%%%%%%%%%%
%%%%%%%%%%%%%%%%%%%%%%%%%%%%%%%%%%%%%%%%%%%%%%%%%%%%%%%%%%%%%%%%%%%%%%%%%%%%%%%%%%%%%%
%%%%%%%%%%%%%%%%%%%%%%%%%%%%%%%%%%%%%%%%%%%%%%%%%%%%%%%%%%%%%%%%%%%%%%%%%%%%%%%%%%%%%%
%%%%%%%%%%%%%%%%%%%%%%%%%%%%%%%%%%%%%%%%%%%%%%%%%%%%%%%%%%%%%%%%%%%%%%%%%%%%%%%%%%%%%%
%%%%%%%%%%%%%%%%%%%%%%%%%%%%%%%%%%%%%%%%%%%%%%%%%%%%%%%%%%%%%%%%%%%%%%%%%%%%%%%%%%%%%%
%%%%%%%%%%%%%%%%%%%%%%%%%%%%%%%%%%%%%%%%%%%%%%%%%%%%%%%%%%%%%%%%%%%%%%%%%%%%%%%%%%%%%%

\section{Preliminaries on functional setting}\label{sec:Preliminaries}

In this section, we present definitions of the function spaces we will consider in our model problem. Furthermore, we briefly explore essential results from the theory of residual minimization in Banach spaces.

%%%%%%%%%%%%%%%%%%%%%%%%%%%%%%%%%%%%%%%%%%%%%%%%%%%%%%%%%%%%%%%%%%%%%%%%%%%%%%%%%%%%%%
%%%%%%%%%%%%%%%%%%%%%%%%%%%%%%%%%%%%%%%%%%%%%%%%%%%%%%%%%%%%%%%%%%%%%%%%%%%%%%%%%%%%%%

    Let $(\X, \norm{\ }_{\X})$ be a normed space. We will denote by $\X^*$ the dual space of $\X$, and will refer to the action of $F \in \X^*$ over elements $x \in \X$, as a duality pairing between $\X^*$ and $\X$. This is,
    \begin{equation}\label{DualityPairing}
        \DProd{F , x}_{\X^*,\X} := F(x).
    \end{equation}
    Additionally, we denote the dual space  $\X^*$ norm as
    \begin{equation}\label{Dual_norm}
        \norm{\, \bullet\, }_{\X^*} := \sup_{x \in \X} \frac{\DProd{\ \bullet\ , x}_{\X^*,\X} }{\norm{x}_{\X}}.
    \end{equation}
    Let $\Omega \subset \mathbb{R}^d$ be a bounded Lipschitz domain, where $d \in \mathbb{N}$ denotes its spatial dimension. For a given $p \geq 1$, consider the Sobolev spaces
    \[
        W^{1,p}(\Omega) := \left\{ v \in L^p(\Omega): \frac{\partial v}{\partial x_i} \in L^p(\Omega), \forall i \in \{1, \cdots, d\} \right\},
    \]
    equipped with the norm
    \[
        \norm{\, \bullet \, }_{W^{1,p}(\Omega)} := \left(\norm{\, \bullet\, }^p_{L^p(\Omega)} + \sum_{i=1}^d \norm{\frac{\partial (\,\bullet\,)}{\partial x_i}}^p_{L^p(\Omega)} \right)^{1/p},
    \]
    where $L^p(\Omega)$ stands for the standard Lebesgue space. In addition, consider the subspace $W_0^{1,p}(\Omega) \subset W^{1,p}(\Omega)$ being the closure of $\C_0^{\infty}(\Omega)$ (the space of smooth functions with compact support on $\Omega$) in the normed space $W^{1,p}(\Omega)$. This is,
    \[
        W_0^{1,p}(\Omega) := \overline{\C_0^{\infty}(\Omega)}^{ \norm{\ }_{W^{1,p}(\Omega)} }.
    \]
    We recall that, as a consequence of Poincaré’s inequality, $\norm{\ }_{W^{1,p}(\Omega)}$ and
    \begin{equation}\label{SemiNormW1p}
        |\,\bullet \,|_{W^{1,p}(\Omega)} := \left( \int_{\Omega} |\nabla (\, \bullet \,)|^p \right)^{1/p}
    \end{equation}
    are equivalent norms in the space $W_0^{1,p}(\Omega)$. 
    %Finally, let $p^*$ be the conjugate exponent of $p$ (i.e., $p^*p = p^* + p$). The dual space of the Sobolev space $W_0^{1,p}(\Omega)$ will be denoted by $W^{-1,p^*}(\Omega)$.

    Since these Sobolev spaces, and their dual spaces, are strictly convex and reflexive for $1 < p < \infty$, they possess favorable geometric properties. A key tool that exploits these properties is the duality map. Although, in general, the duality map on a Banach space $\X$ is defined as a multivalued mapping from $\X$ to $2^{\X^*}$, when $\X^*$ is strictly convex it becomes a single-valued map (see \cite[Prop. 12.3]{Deimling2013Nonlinear}). This property allows us to use the following definition of the duality map, which we will consider in this paper.
    
    \begin{Def}\label{DmapDef}
        Let $\X$ be a Banach space such that $\X^*$ is strictly convex. For $p > 1$, the \textbf{duality map} $J_{p,\X} : \X \to \X^*$ is the unique operator that satisfies:
    \begin{enumerate}[i.]
        \item $\DProd{ J_{p,\X} (x), x}_{\X^*,\X} = \norm{J_{p,\X}(x)}_{\X^*} \norm{x}_{\X}$.
        \item $\norm{J_{p,\X} (x)}_{\X^*} = \norm{x}^{p-1}_{\X}$.
    \end{enumerate}
    \end{Def}
    The existence of $J_{p,\X}$ is guaranteed by the Hahn--Banach extension theorem, and the uniqueness is due to the strict convexity of $\X^*$. The following result considers a distinctive feature of the duality map.
    \begin{Prop}\label{DMapFromTheNorm}
        Let $\X$ be a strictly convex Banach space and consider $p > 1$. Let $\phi : \X \to \mathbb{R}$ be defined as $\phi(\, \bullet \,) := \frac{1}{p} \norm{\, \bullet \,}_{\X}^p$. Then, $\phi$ is a G\^ateaux differentiable functional for all $x \in \X$, and its derivative at $x \in \X$ satisfies the following identity:
\begin{equation}\label{DmapAsGradient}
            \nabla \phi(x) = J_{p,\X}(x).
        \end{equation}
    \end{Prop}
    \Cref{DMapFromTheNorm} is very useful for determining the duality map $J_{p,\X}$ of a Banach space $\X$ from its norm. {For a detailed proof, we refer the reader to \cite[Th. 4.4]{Cioranescu1990Geometry}, which establishes this relation for generalized duality maps by considering the weight function $\varphi(t) = t^{p-1}$.}
    
   The next result provides a particular case for which the bijectivity of the duality map is ensured. {For a proof, see, for example,  \cite[Prop. 4.7]{Cioranescu1990Geometry}.}
    \begin{Prop}\label{DmapBijective}
        Let $\X$ and $\X^*$ be strictly convex and reflexive Banach spaces, and let $J_{p,\X} : \X \to \X^*$ be the duality map on $\X$ for $p > 1$. Then the following statements hold:
        \begin{enumerate}[(a)]
            \item The duality map $J_{p,\X}$ is bijective. %and $J_{p,\X^*} : \X^* \to \X^{**}$ are bijective.

            \item If $\I_{\X}: \X \to \X^{**}$ denotes the canonical injection and $p^* = p/(p-1)$, then
\begin{equation}\label{DualityMapDualSpace}
                    J_{p^*,\X^*} = \I_{\X} \circ J_{p,\X}^{-1}.
                \end{equation}
        \end{enumerate}
    \end{Prop}
    
A last ingredient we require for the subsequent analysis is the duality map of a linear subspace. 

    \begin{Prop}\label{PropDmapSubspace}
        Let $\X$ be a Banach space such that $\X^*$ is strictly convex, and let $J_{p,\X} : \X \to \X^*$ be the duality map on $\X$. Let $\M \subset \X$ be a linear subspace of $\X$ and $p > 1$. Then, the duality map $J_{p,\M} : \M \to \M^*$ verify the following identity:
        \begin{align}\label{DmapLinearSubspace}
            I_{\M}^* \circ J_{p,\X} \circ I_{\M} = J_{p,\M},
        \end{align}
        where $I_{\M} : \M \to \X$ is the natural injection.
    \end{Prop}

    \begin{proof}
        The proof for $p=2$ can be found in \cite[Lem. 2.3]{Muga2020Discretization}. For $p>1$, it follows similarly.
    \end{proof}

%%%%%%%%%%%%%%%%%%%%%%%%%%%%%%%%%%%%%%%%%%%%%%%%%%%%%%%%%%%%%%%%%%%%%%%%%%%%%%%%%%%%%%
%%%%%%%%%%%%%%%%%%%%%%%%%%%%%%%%%%%%%%%%%%%%%%%%%%%%%%%%%%%%%%%%%%%%%%%%%%%%%%%%%%%%%%

%\subsection{Best approximation in Banach spaces}

    We now present a proposition establishing conditions for the existence and uniqueness of the best approximation from a Banach space to a nonempty, closed, and convex subset, along with its corresponding a priori bounds.
    
    \begin{Prop}[Best approximation and a priori bounds]\label{BestApproximation}\label{APrioriBoundsBestApproximation}
        Let $\X$ be a reflexive Banach space and let $\M \subset \X$ be a nonempty closed convex subset. Then, for every $x \in \X$, there exists $\hat{x} \in \M$ such that
        \[
            \norm{x - \hat{x}}_{\X} = \inf_{y \in \M} \norm{x - y}_{\X}.
        \]
        Furthermore, if $\X$ is strictly convex, then $\hat{x}$ is unique. {Additionally, if $0_{\X} \in \M$, the best approximation $\hat{x}$ satisfies the a priori bound:
        \[
            \norm{\hat{x}}_{\X} \leq 2 \norm{x}_{\X}.
        \]
        Moreover, if $\M$ is a closed linear subspace, this bound can be sharpened to:
        \[
            \norm{\hat{x}}_{\X} \leq C_{\mathtt{BM}}(\X) \norm{x}_{\X},
        \]
        where $C_{\mathtt{BM}}(\X) \in [1,2]$ is the Banach--Mazur constant of $\X$.}
    \end{Prop}

    \begin{Rem}\label{Banach_Mazur_constant}
        For the precise statement regarding the conditions for the existence and uniqueness of the best approximation, we refer   to \cite[Problem 5.14-4]{Ciarlet2013Linear}.
        The derivation of the standard a priori bounds can be found in \cite[Prop. 3.5]{Muga2020Discretization}. It is worth noting that standard bounds in general Banach spaces can be pessimistic. To obtain the sharpened estimate involving the Banach--Mazur constant $C_{\mathtt{BM}}(\X)$, it is necessary to exploit specific geometric properties of $\X$, particularly its proximity to a Euclidean space. For a formal definition of this constant and further geometric insights, see \cite[Def. 2]{Stern2015Banach}.
    \end{Rem}

\section{Residual minimization in dual norms}\label{sec:MinResDualNorms}

    Let us consider a finite-dimensional subspace $\U_h \subset \U$. The residual minimization problem applied to problem \eqref{abstractPDE} is as follows: given a $F \in \V^*$ and $p > 1$, find $u_h \in \U_h$ such that
    \begin{align}\label{MinResProblem}
        u_h = \arg\min_{w_h \in \U_h} \frac{1}{p^*} \norm{F - A(w_h)}^{p^*}_{\V^*},
    \end{align}
    where $p^*$ is such that $p^*p = p^* + p$. When analyzing problem \eqref{MinResProblem}, we assume the existence and uniqueness of a solution to the abstract problem \eqref{abstractPDE}.

%%%%%%%%%%%%%%%%%%%%%%%%%%%%%%%%%%%%%%%%%%%%%%%%%%%%%%%%%%%%%%%%%%%%%%%%%%%%%%%%%%%%%%
%%%%%%%%%%%%%%%%%%%%%%%%%%%%%%%%%%%%%%%%%%%%%%%%%%%%%%%%%%%%%%%%%%%%%%%%%%%%%%%%%%%%%%
    For the problem \eqref{MinResProblem} to be well-posed, the following assumptions are established.
    \begin{Assum}\label{WellPosedness}
        In problem \eqref{MinResProblem}, we assume the following statements:
        \begin{enumerate}[1.]
            \item \textbf{Existence:} $\V^*$ is reflexive and $A(\U_h)$ is a nonempty closed convex subset of $\V^*$. \label{WellPosedness1}
            \item \textbf{Uniqueness:} $\V^*$ is strictly convex and $A : \U \to \V^*$ is injective. \label{WellPosedness2}
            \item \textbf{Stability:} There exists $\alpha > 0$ such that $\norm{w}_{\U} \lesssim \norm{A(w)}_{\V^*}^{\alpha}$ for all $w \in \U$. \label{WellPosednessStability}
        \end{enumerate}
    \end{Assum}
    
    Indeed, assuming the above statements, from \Cref{BestApproximation} it follows that there exists a unique $\hat{F}_h \in A(\U_h)$ such that
    \begin{align}\label{BestApproximationResidual}
        \norm{F - \hat{F}_h}_{\V^*} = \min_{G_h \in A(\U_h)} \norm{F - G_h}_{\V^*}.
    \end{align}
    In addition, since $A : \U \to \V^*$ is injective, then there exists a unique $u_h \in \U_h$ such that $A(u_h) = \hat{F}_h$. Thus, the existence and uniqueness of a solution to problem \eqref{MinResProblem} is guaranteed. Hypothesis \eqref{WellPosednessStability} is assumed to ensure the stability of the method, since \Cref{APrioriBoundsBestApproximation} shows that
    \begin{align*}
        \norm{u_h}_{\U} \leq \norm{A(u_h)}_{\V^*}^{\alpha} \leq 2^{\alpha} \norm{F}_{\V^*}^{\alpha}.
    \end{align*}

\begin{comment} 
    An example that satisfies the above assumptions is presented below.
    
    \begin{Ex}\label{WellPosednessExample}
        Let us consider $\U = \V = W_0^{1,p}(\Omega)$ with $1 < p < \infty$, and let $A := \U \to \V^*$ be the $p$-Laplacian operator defined as
        \begin{align*}
            \DProd{A(u) , v}_{\V^*,\V} := \int_{\Omega} |\nabla u|^{p-2} \nabla u \cdot \nabla v\dx, \quad \forall u, v \in W_0^{1,p}(\Omega).
        \end{align*}
        It is well known that $W_0^{1,p}(\Omega)$ satisfies statement \eqref{WellPosedness1}. Furthermore, it can be proven that $A$ coincides with the duality map $J_{p,\V}$, which is bijective. Thus, statement \eqref{WellPosedness2} is satisfied. Moreover, note that $J_{p,\U_h} = A\bigr|_{\U_h}$ (see \cite[Lemma 2.3]{Muga2020Discretization}). Therefore, from \Cref{DmapBijective}, it follows that $J_{p,\U_h} = A\bigr|_{\U_h}$ is bijective, and $A(\U_h) = (\U_h)^*$ satisfies statement \eqref{WellPosedness3}. Finally, from \Cref{DmapDef}, we have that
        \begin{align*}
            \norm{w}_{\U} = \norm{A(w)}_{\V^*}^{1/(p-1)}, \quad \forall w \in \U.
        \end{align*}
        Thus, the statement \eqref{WellPosednessStability} is also satisfied.
    \end{Ex}
\end{comment}

\subsection{Useful characterizations}

    Herein, we provide a characterization of the solution obtained by the MinRes method, which is a central result of this work.

    \begin{Theo}[Problem equivalence]\label{MinResCharacterization}
        Let $A : \U \to \V^*$ be a Fréchet differentiable operator where $\V$ and $\V^*$ are strictly convex and reflexive Banach spaces. Let $\U_h$ be a finite-dimensional subspace of $\U$. Given $F \in \V^*$, the statements listed below are equivalent:
        \begin{enumerate}
            \item $u_h$ is a solution to the minimization problem \eqref{MinResProblem}. \label{MinResCharacterization1}
            \item For all $p>1$, $u_h \in \U_h$ is a solution of the problem:
                \begin{align}\label{PetrovGalerkinMixedFormulation}
                    \DProd{\nabla A(u_h)w_h , J_{p,\V}^{-1}(F - A(u_h))}_{\V^{*},\V} = 0, \quad \forall w_h \in \U_h.
                \end{align} \label{MinResCharacterization2}
            \item For all $p>1$, there exists a residual representative $r \in \V$, such that $(r , u_h) \in \V \times \U_h$ solves the semi-infinite mixed formulation:
                \begin{align}\label{MinResMixedFormulation}
                    \left\{
                    \begin{aligned}
                        &\DProd{J_{p,\V}(r) , v}_{\V} + \DProd{A(u_h) , v}_{\V^*,\V} &&= \DProd{F , v}_{\V^*,\V}, && \quad \forall v \in \V, \\
                        &\DProd{ \nabla A(u_h)w_h , r}_{\V^*,\V} &&= 0, && \quad \forall w_h \in \U_h.
                    \end{aligned}
                    \right.
                \end{align}\label{MinResCharacterization3}
        \end{enumerate}
    \end{Theo}

    \begin{proof}
        Let us consider the functional $\F : \U_h \to \mathbb{R}$ such that
        \[
            \F(w_h) := \frac{1}{p^*}\norm{F - A(w_h)}^{p^*}_{\V^*}.
        \]

        Thus, we can rewrite the problem \eqref{MinResProblem} as: find $u_h \in \U_h$ such that
        \[
            u_h = \arg \min_{w_h \in \U_h} \F(w_h).
        \]

        From \eqref{DmapAsGradient}, we arrive at the following expression:
        \begin{align*}
            \nabla \F(u_h)w_h &= \frac{\mathrm{d}}{\mathrm{d}t} \F(u_h + tw_h) \Bigr|_{t=0} = \frac{\mathrm{d}}{\mathrm{d}t} \frac{1}{p^*}\norm{F - A(u_h + tw_h)}^{p^*}_{\V^*} \Bigr|_{t=0} \\
                &= -J_{p^*,\V^*}(F - A(u_h + tw_h))\nabla A(u_h + tw_h)w_h \Bigr|_{t=0} \\
                &= -J_{p^*,\V^*}(F - A(u_h))\nabla A(u_h)w_h.
        \end{align*}

        Writing the above with a duality product and taking into account the identity \eqref{DualityMapDualSpace}, we obtain that
        \begin{align*}
            \DProd{\nabla \F(u_h) , w_h}_{\U^*,\U} &= - \DProd{J_{p^*,\V^*}(F - A(u_h)) , \nabla A(u_h)w_h}_{\V^{**},\V^*} \\
            &= - \DProd{\left(\I_{\V} \circ J_{p,\V}^{-1}\right)(F - A(u_h)) , \nabla A(u_h)w_h}_{\V^{**},\V^*} \\
            &= - \DProd{\nabla A(u_h)w_h , J_{p,\V}^{-1}(F - A(u_h))}_{\V^{*},\V}.
        \end{align*}

        So, after setting $r := J_{p,\V}^{-1}(F - A(u_h)) \in \V$, it follows that \eqref{MinResCharacterization1} $\Leftrightarrow$ \eqref{MinResCharacterization2} $\Leftrightarrow$ \eqref{MinResCharacterization3}.
    \end{proof}

    Note that, from the definition of the duality map $J_{p,\V}$ and taking into account that $J_{p,\V}(r) := F - A(u_h)$, it follows that:
    \[
        \norm{F - A(u_h)}_{\V^*} = \norm{r}_{\V}^{p-1}.
    \]
    The following subsection presents an a posteriori error estimate based on this residual representative.

%%%%%%%%%%%%%%%%%%%%%%%%%%%%%%%%%%%%%%%%%%%%%%%%%%%%%%%%%%%%%%%%%%%%%%%%%%%%%%%%%%%%%%

\subsection{A continuous a posteriori error estimator}\label{AposterioriErrorEstimator}

In this section, we derive reliability and efficiency estimates for the residual minimization method applied to the $p$-Laplacian problem for $p \in (1,\infty)$.
We begin by stating some fundamental properties of the $p$-Laplacian operator, which are proven in \cite[Sect. 5]{glowinski1975sur}.

\begin{Lem}[Properties of the $p$-Laplacian]\label{PropertiesPLaplacian}
    Let $\U = \V = W_0^{1,p}(\Omega)$ and let $A: \V \to \V^*$ be the $p$-Laplacian operator defined as
    \begin{align}\label{pLaplacianOperator}
        \DProd{A(u) , v}_{\V^*,\V} := \int_{\Omega} |\nabla u|^{p-2} \nabla u \cdot \nabla v\dx, \quad \forall u, v \in \V.
    \end{align}
    Then, for $p > 1$, there exist positive constants $C_{\mathtt{L}}, c_{\mathtt{L}}, C_{\mathtt{M}}, c_{\mathtt{M}}$ that depend on $p$, such that for all $u, w \in \V$, the operator $A$ satisfies:
    \begin{enumerate}[(i)]
        \item \textbf{Continuity:}
        \begin{align}\label{ContinuityA}
            \norm{A(u) - A(w)}_{\V^*} \leq 
            \begin{cases} 
                C_{\mathtt{L}} \left( \norm{u}_{\V} + \norm{w}_{\V} \right)^{p-2} \norm{u - w}_{\V} & \text{if } p \geq 2, \\
                c_{\mathtt{L}} \norm{u - w}_{\V}^{p-1} & \text{if } 1 < p < 2.
            \end{cases}
        \end{align}
        
        \item \textbf{Monotonicity:}
        \begin{align}\label{MonotonicityA}
            \DProd{A(u) - A(w) , u - w}_{\V^*,\V} \geq \begin{cases} 
                C_{\mathtt{M}} \norm{u - w}_{\V}^{p} & \text{if } p \geq 2, \\
        \displaystyle        {c_{\mathtt{M}} \left( \norm{u}_{\V} + \norm{w}_{\V} \right)^{p-2} \norm{u - w}_{\V}^2} & \text{if } 1 < p < 2.
            \end{cases}
        \end{align}
    \end{enumerate}
\end{Lem}

Additionally, the following lemma will be used to obtain several results presented in this paper applied to the $p$-Laplacian problem.

\begin{Lem}[Auxiliary bounds]\label{LemStabilityBounds}
    Let $\U = \V = W_0^{1,p}(\Omega)$ with $p \in (1, \infty)$ and let $u$ be the solution to the abstract problem \eqref{abstractPDE}, where $A:\U \to \V^*$ is the $p$-Laplacian operator defined in \eqref{pLaplacianOperator}. Let $u_h \in \U_h$ be the solution to the MinRes problem \eqref{MinResProblem}, and let $\hat{u}_h \in \U_h$ be the best approximation to $u$. Then, the following bounds hold:
    \begin{enumerate}[(i)]
        \item $\norm{u}_{\V} \leq \norm{F}_{\V^*}^{\frac{1}{p-1}}$. \label{LemStabilityBoundsItem1}
        \item $\norm{u_h}_{\V} \leq 2^{\frac{1}{p-1}} \norm{F}_{\V^*}^{\frac{1}{p-1}}$. \label{LemStabilityBoundsItem2}
        \item $\norm{\hat{u}_h}_{\V} \leq C_{\mathtt{BM}}(\V) \norm{F}_{\V^*}^{\frac{1}{p-1}}$, where $C_{\mathtt{BM}}(\V)$ is the Banach--Mazur constant of $\V$ {introduced in \Cref{Banach_Mazur_constant}}. \label{LemStabilityBoundsItem3}
    \end{enumerate}
\end{Lem}

\begin{proof}
    {Directly from the definition of the operator in \eqref{pLaplacianOperator}, we have the identity
    \begin{align}\label{Coercivity_pLaplacian}
        \DProd{A(v), v}_{\V^*,\V} = \norm{v}_{\V}^p \quad \forall v \in \V.
    \end{align}
    
    Item \eqref{LemStabilityBoundsItem1} follows from applying \eqref{Coercivity_pLaplacian} to the exact solution: $\norm{u}_{\V}^p = \DProd{A(u),u}_{\V^*,\V} \leq \norm{F}_{\V^*} \norm{u}_{\V}$.
    
    For \eqref{LemStabilityBoundsItem2}, since $0 \in \U_h$ and $A(0)=0$, the minimization property guarantees $\norm{F - A(u_h)}_{\V^*} \leq \norm{F - A(0)}_{\V^*} = \norm{F}_{\V^*}$. By the triangle inequality, $\norm{A(u_h)}_{\V^*} \leq 2 \norm{F}_{\V^*}$. Using \eqref{Coercivity_pLaplacian} for $u_h$, we have $\norm{u_h}_{\V}^p = \DProd{A(u_h),u_h}_{\V^*,\V} \leq \norm{A(u_h)}_{\V^*} \norm{u_h}_{\V}$. Dividing by $\norm{u_h}_{\V}$ gives $\norm{u_h}_{\V}^{p-1} \leq \norm{A(u_h)}_{\V^*} \leq 2 \norm{F}_{\V^*}$, which directly implies \eqref{LemStabilityBoundsItem2}.
    
    Finally, from \Cref{APrioriBoundsBestApproximation}, we know that $\norm{\hat{u}_h}_{\V} \leq C_{\mathtt{BM}}(\V) \norm{u}_{\V}$. Combining this with \eqref{LemStabilityBoundsItem1} implies \eqref{LemStabilityBoundsItem3}.}
\end{proof}

Based on these properties, we establish the following   error estimator.

\begin{Theo}[Continuous a posteriori error estimator]\label{ThmAposterioriError}
    Let $u$ and $u_h$ be the solutions of the abstract problem \eqref{abstractPDE} and the MinRes problem \eqref{MinResProblem}, respectively. Let $r \in \V$ be the residual representative defined in \Cref{MinResCharacterization}. Then, the following global error estimates hold:
    \begin{align}\label{FinalErrorEstimate}
        \begin{cases}
            C_1 \norm{r}_{\V}^{p-1} \leq \norm{u-u_h}_{\V} \leq C_2 \norm{r}_{\V} & \text{if } p \geq 2, \\
            c_1 \norm{r}_{\V} \leq \norm{u-u_h}_{\V} \leq c_2 \norm{r}_{\V}^{p-1} & \text{if } 1 < p < 2,
        \end{cases}
    \end{align}
    where
    \begin{align*}
        C_1 &= C_{\mathtt{L}}^{-1}(1 + C_{\mathtt{BM}}(\V))^{2-p} \norm{F}_{\V^*}^{-\frac{p-2}{p-1}}; & C_2 &= C_{\mathtt{M}}^{-\frac{1}{p-1}}; \\
        c_1 &= c_{\mathtt{L}}^{-\frac{1}{p-1}}; & c_2 &= c_{\mathtt{M}}^{-1} \left( 1 + 2^{\frac{1}{p-1}} \right)^{2-p} \norm{F}_{\V^*}^{\frac{2-p}{p-1}}.
    \end{align*}
\end{Theo}

\begin{proof}
    For $p \geq 2$, we first prove the lower bound. Let $\hat{u}_{h}$ be the best aproximation of $u$ in $\U_h$. Using the optimality of $u_h$ in \eqref{MinResProblem} and the Lipschitz continuity property \eqref{ContinuityA}, we obtain:
    \begin{align*}
        \norm{r}_{\V}^{p-1} = \norm{F - A(u_h)}_{\V^*} \leq \norm{F - A(\hat{u}_h)}_{\V^*} \leq C_{\mathtt{L}} \left( \norm{u}_{\V} + \norm{\hat{u}_h}_{\V} \right)^{p-2} \norm{u - \hat{u}_h}_{\V}.
    \end{align*}
    Invoking \Cref{LemStabilityBounds} and noting $\norm{u - \hat{u}_h}_{\V} \leq \norm{u - u_h}_{\V}$, we get the lower bound. For the upper bound, using the strong monotonicity \eqref{MonotonicityA} {and the fact that $J_{p,\V}(r) = F - A(u_h)$}, we get:
    \begin{align*}
        C_{\mathtt{M}} \norm{u - u_h}_{\V}^{p} &\leq \DProd{A(u) - A(u_h) , u - u_h}_{\V^*,\V} \leq \norm{J_{p,\V}(r)}_{\V^*} \norm{u - u_h}_{\V} \\
        & = \norm{r}_{\V}^{p-1} \norm{u - u_h}_{\V}.
    \end{align*}
    Simplifying gives the upper bound.
    For $1 < p < 2$, the lower bound follows directly from the definition of the residual and the Hölder continuity \eqref{ContinuityA}:
    \begin{align*}
        \norm{r}_{\V}^{p-1} = \norm{A(u) - A(u_h)}_{\V^*} \leq c_{\mathtt{L}} \norm{u - u_h}_{\V}^{p-1} \implies c_{\mathtt{L}}^{-\frac{1}{p-1}} \norm{r}_{\V} \leq \norm{u - u_h}_{\V}.
    \end{align*}
    For the upper bound, using the degenerate strong monotonicity \eqref{MonotonicityA}, we get:
    \begin{align*}
        c_{\mathtt{M}} \left( \norm{u}_{\V} + \norm{u_h}_{\V} \right)^{p-2} \norm{u - u_h}_{\V}^2 &\leq \DProd{A(u) - A(u_h) , u - u_h}_{\V^*,\V} 
        \leq \norm{r}_{\V}^{p-1} \norm{u - u_h}_{\V}.
    \end{align*}
    Dividing by $\norm{u - u_h}_{\V}$ and bounding the weight term using \Cref{LemStabilityBounds} \eqref{LemStabilityBoundsItem1} and \eqref{LemStabilityBoundsItem2} yields the result.
\end{proof}

%%%%%%%%%%%%%%%%%%%%%%%%%%%%%%%%%%%%%%%%%%%%%%%%%%%%%%%%%%%%%%%%%%%%%%%%%%%%%%%%%%%%%%
%%%%%%%%%%%%%%%%%%%%%%%%%%%%%%%%%%%%%%%%%%%%%%%%%%%%%%%%%%%%%%%%%%%%%%%%%%%%%%%%%%%%%%
%%%%%%%%%%%%%%%%%%%%%%%%%%%%%%%%%%%%%%%%%%%%%%%%%%%%%%%%%%%%%%%%%%%%%%%%%%%%%%%%%%%%%%
%%%%%%%%%%%%%%%%%%%%%%%%%%%%%%%%%%%%%%%%%%%%%%%%%%%%%%%%%%%%%%%%%%%%%%%%%%%%%%%%%%%%%%
%%%%%%%%%%%%%%%%%%%%%%%%%%%%%%%%%%%%%%%%%%%%%%%%%%%%%%%%%%%%%%%%%%%%%%%%%%%%%%%%%%%%%%
%%%%%%%%%%%%%%%%%%%%%%%%%%%%%%%%%%%%%%%%%%%%%%%%%%%%%%%%%%%%%%%%%%%%%%%%%%%%%%%%%%%%%%
%%%%%%%%%%%%%%%%%%%%%%%%%%%%%%%%%%%%%%%%%%%%%%%%%%%%%%%%%%%%%%%%%%%%%%%%%%%%%%%%%%%%%%
%%%%%%%%%%%%%%%%%%%%%%%%%%%%%%%%%%%%%%%%%%%%%%%%%%%%%%%%%%%%%%%%%%%%%%%%%%%%%%%%%%%%%%
%%%%%%%%%%%%%%%%%%%%%%%%%%%%%%%%%%%%%%%%%%%%%%%%%%%%%%%%%%%%%%%%%%%%%%%%%%%%%%%%%%%%%%
%%%%%%%%%%%%%%%%%%%%%%%%%%%%%%%%%%%%%%%%%%%%%%%%%%%%%%%%%%%%%%%%%%%%%%%%%%%%%%%%%%%%%%
%%%%%%%%%%%%%%%%%%%%%%%%%%%%%%%%%%%%%%%%%%%%%%%%%%%%%%%%%%%%%%%%%%%%%%%%%%%%%%%%%%%%%%
%%%%%%%%%%%%%%%%%%%%%%%%%%%%%%%%%%%%%%%%%%%%%%%%%%%%%%%%%%%%%%%%%%%%%%%%%%%%%%%%%%%%%%
%%%%%%%%%%%%%%%%%%%%%%%%%%%%%%%%%%%%%%%%%%%%%%%%%%%%%%%%%%%%%%%%%%%%%%%%%%%%%%%%%%%%%%
%%%%%%%%%%%%%%%%%%%%%%%%%%%%%%%%%%%%%%%%%%%%%%%%%%%%%%%%%%%%%%%%%%%%%%%%%%%%%%%%%%%%%%
    
\section{Residual minimization in discrete dual norms}\label{sec:discreteMinRes}
    A complication of the formulation \eqref{MinResProblem} is that the dual norm in practice may not be computable, which happens when the space $\V$ is intractable (e.g., when $\V = W_0^{1,p}(\Omega)$). One way to solve this complication is to restrict the supremum to a discrete normed space ($\V_h , \norm{\ }_{\V_h}$) which may be non-conforming (i.e. $\V_h \not\subset \V$). For this modification to make sense with a non-conforming space $\V_h$, we will assume the following:
    \begin{Assum}\label{DiscreteAssumptions}
        Let $\V_h$ be a finite-dimensional normed space and let \mbox{$\V(h) := \V + \V_h$} be endowed with a norm denoted by $\norm{\ }_h$. We assume the following:
        \begin{enumerate}
            \item $\norm{v}_{h} \cong \norm{v}_{\V}$ and $\norm{v_h}_{h} \cong \norm{v_h}_{\V_h}$ for all $v \in \V$ and $v_h \in \V_h$.
            \item There exists 
            %an operator 
            $A_h : \U \to \V(h)^*$ such that $A_h(w)|_{\V} = A(w)$ for all $w \in \U$.
        \end{enumerate}
    \end{Assum}
    
    In this way, we can define the following discrete MinRes problem: given a $F_h \in \V(h)^*$ and $p > 1$, find $u_h \in \U_h$ such that
    \begin{align}\label{DiscreteMinResProblem}
        u_h = \arg\min_{w_h \in \U_h} \frac{1}{p^*} \norm{I_h^*(F_h - A_h(w_h))}^{p^*}_{(\V_h)^*},
    \end{align}
    where $I_h : \V_h \to \V(h)$ is the natural injection and
    \begin{align*}
        \norm{\, \bullet \,}_{(\V_h)^*} := \sup_{v_h \in \V_h} \frac{\DProd{\, \bullet \, , v_h}_{\V(h)^*,\V(h)}}{\norm{v_h}_{h}}. 
    \end{align*}
    Here the notation $(\V_h)^*$ is used to avoid confusion with a discrete subspace $\V_h^* \subset \V^*$.
    
%%%%%%%%%%%%%%%%%%%%%%%%%%%%%%%%%%%%%%%%%%%%%%%%%%%%%%%%%%%%%%%%%%%%%%%%%%%%%%%%%%%%%%
%%%%%%%%%%%%%%%%%%%%%%%%%%%%%%%%%%%%%%%%%%%%%%%%%%%%%%%%%%%%%%%%%%%%%%%%%%%%%%%%%%%%%%
    For the well-posedness of problem \eqref{DiscreteMinResProblem}, we establish the same assumptions as in the exact method (\Cref{WellPosedness}), but replacing the space $\V$ by $\V(h)$ and the operator $A$ by $A_h$. However, in this case, we need to add an additional hypothesis to ensure the solution is unique. Thus, we have the following assumptions.
    
    \begin{Assum}\label{DiscreteWellPosedness}
        In problem \eqref{DiscreteMinResProblem}, we assume the following statements:
        \begin{enumerate}
            \item $\U$ and $\U_h$ are subspaces of $\V(h)$ and $\V_h$, respectively. \label{WellPosedness5}
            \item $A_h$ is a strictly monotone operator, i.e.,
            \[
                \DProd{A_h(u) - A_h(w), u - w}_{\V(h)^*,\V(h)} > 0, \quad \forall u, w \in \U, \ u \neq w.\]
        %\end{align*}%\label{WellPosedness6}
        \end{enumerate}
    \end{Assum}
\Cref{DiscreteWellPosedness} is necessary since there might exist some $\tilde{u}_h \neq u_h$ in $\U_h$ such that
    \begin{align}\label{NonUniquenessDiscrete}
        \DProd{A_h(u_h) - A_h(\tilde{u}_h), v_h}_{(\V_h)^*,\V_h} = 0, \quad \forall v_h \in \V_h,
    \end{align}
    but $A_h(u_h) \neq A_h(\tilde{u}_h)$ in $\V(h)^*$. But, if $A_h$ is strictly monotone, setting $v_h := u_h - \tilde{u}_h$ in \eqref{NonUniquenessDiscrete}, it follows that $u_h = \tilde{u}_h$. Furthermore, note that $\U \subset \V(h)$ and $\U_h \subset \V_h$ are necessary to make sense of the duality products in \eqref{NonUniquenessDiscrete}.

\begin{comment}
    Let us consider an example that satisfies the above assumptions.

    \begin{Ex}\label{DiscreteWellPosednessExample}
        Let $\Omega$ be a polygonal domain in $\mathbb{R}^2$ and let $\T^h$ be a triangulation of $\Omega$. Let us consider $\U = \V = W_0^{1,p}(\Omega)$ with $1 < p < \infty$, $\V_h$ the Crouzeix--Raviart finite element space, and $\V(h) := \V + \V_h$ the Banach space endowed with the broken $W^{1,p}$-norm
        \begin{align}\label{BrokenWpNorm}
            \norm{\, \cdot \,}_h^p := \sum_{T \in \T^h} |\, \cdot \,|_{W^{1,p}(T)}.
        \end{align}
        Let $A_h : \U \to \V(h)^*$ be the extended $p$-Laplacian operator defined as
        \begin{align*}
            \DProd{A_h(u) , v_h}_{\V(h)^*,\V_h} := \sum_{T \in \T^h} \int_{T} |\nabla u|^{p-2} \nabla u \cdot \nabla v_h\dx, \quad \forall u \in \U, \ v_h \in \V(h).
        \end{align*}
        It can be proven that $A_h$ and $\V(h)$ satisfy all the hypotheses of \Cref{WellPosedness} with $\alpha = \frac{1}{p-1}$. Furthermore, Assumptions \ref{DiscreteAssumptions} and \ref{DiscreteWellPosedness} are trivially satisfied for any subspace $\U_h$ of $\V_h$ because $A_h$ coincides with the duality map of $(\V(h), \norm{\, \cdot \,}_h)$, which is strictly monotone if and only if $\V(h)$ is strictly convex (see \cite[Theorem 12.2]{Deimling2013Nonlinear}).
    \end{Ex}
\end{comment}

%%%%%%%%%%%%%%%%%%%%%%%%%%%%%%%%%%%%%%%%%%%%%%%%%%%%%%%%%%%%%%%%%%%%%%%%%%%%%%%%%%%%%%
%%%%%%%%%%%%%%%%%%%%%%%%%%%%%%%%%%%%%%%%%%%%%%%%%%%%%%%%%%%%%%%%%%%%%%%%%%%%%%%%%%%%%%
    
\subsection{Discrete characterizations}

    Similarly to the exact method, a result that characterizes the solution to problem \eqref{DiscreteMinResProblem} is presented in the following theorem.

    \begin{Theo}[Discrete problem equivalence]\label{DiscreteMinResCharacterization}
        Let $A_h : \U \to \V(h)^*$ be a Fréchet differentiable operator. Assume that $\V$ and $\V^*$ are strictly convex and reflexive Banach spaces, and that \Cref{DiscreteAssumptions} holds true. Let $\U_h$ be a discrete subspace of $\U$ and let $\V_h$ be a finite-dimensional normed space. Given $F_h \in \V(h)^*$, the statements listed below are equivalent:
        \begin{enumerate}[i.]
            \item $u_h$ is a solution to the minimization problem \eqref{DiscreteMinResProblem}. \label{DiscreteMinResCharacterization1}
            \item For all $p>1$, $u_h \in \U_h$ is a solution of the problem:
                \begin{align}\label{DiscretePetrovGalerkinMixedFormulation}
                    \DProd{\nabla A_h(u_h)w_h \, , \, I_h J_{p,\V_h}^{-1} \circ I^*_h(F_h - A_h(u_h))}_{\V(h)^{*},\V(h)} = 0, \quad \forall w_h \in \U_h.
                \end{align}
                \label{DiscreteMinResCharacterization2}
            \item For all $p>1$, there exists a residual representative $r_h \in \V_h$, such that $(r_h , u_h) \in \V_h \times \U_h$ solves the finite-dimensional  mixed formulation:
                %\begin{empheq}[left=\empheqlbrace]{align}
                \begin{subequations}
                \begin{alignat}{2}
                    &\DProd{J_{p,\V(h)}(r_h) , v_h}_{\V(h)^*, \V(h)} + \DProd{A_h(u_h) , v_h}_{\V(h)^*,\V(h)} &&= \DProd{F_h , v_h}_{\V(h)^*,\V(h)}, \label{DiscreteMinResMixedFormulation1} \\
                    &\DProd{ \nabla A_h(u_h)w_h , r_h}_{\V(h)^*,\V(h)} &&= 0, \label{DiscreteMinResMixedFormulation2}
                \end{alignat}
            \end{subequations}
                for all $(v_h,w_h) \in \V_h \times \U_h$. \label{DiscreteMinResCharacterization3}
        \end{enumerate}
    \end{Theo}
    \begin{proof}
        Analogously to the proof of \Cref{MinResCharacterization}, we consider the functional $\F_h : \U_h \to \mathbb{R}$ defined by
        \[
            \F_h(w_h) := \frac{1}{p^*}\norm{I^*_h\bigl(F_h - A_h(w_h)\bigr)}^{p^*}_{(\V_h)^*}.
        \]
        Using \eqref{DmapAsGradient}, we derive its gradient as follows:
        \begin{align}\label{GradientDiscreteMinRes}
            \DProd{\nabla \F_h(u_h) , w_h}_{\U_h^*,\U_h} &= - \DProd{\!\nabla A_h(u_h)w_h \, , \, I_h J_{p,\V_h}^{-1}\!\! \circ I^*_h\bigl(F_h - A_h(u_h)\bigr)}_{\V(h)^{*},\V(h)}.
        \end{align}
        Since $u_h$ minimizes $\F_h$ if and only if the gradient vanishes for all $w_h \in \U_h$, this establishes the equivalence \eqref{DiscreteMinResCharacterization1} $\Leftrightarrow$ \eqref{DiscreteMinResCharacterization2}.
        
        Finally, we define the discrete residual representative $r_h := J_{p,\V_h}^{-1} \circ I^*_h\bigl(F_h - A_h(u_h)\bigr) \in \V_h$. By applying the duality map $J_{p,\V_h}$ to both sides and invoking \Cref{PropDmapSubspace} (which states that $J_{p,\V_h} = I_h^* J_{p,\V(h)} I_h$), we obtain the following identity 
        \begin{align*}
            I_{h}^* \circ J_{p,\V(h)} (I_{h} r_h) = I^*_h\bigl(F_h - A_h(u_h)\bigr),
        \end{align*}
        in $(\V_h)^*$.    This equality is equivalent to the variational statement:
        \begin{align*}
            \DProd{J_{p,\V(h)}(r_h) , v_h}_{\V(h)^*, \V(h)} = \DProd{F_h - A_h(u_h) , v_h}_{\V(h)^*,\V(h)}, \quad \forall v_h \in \V_h,
        \end{align*}
        which corresponds exactly to \eqref{DiscreteMinResMixedFormulation1}. Combining this with the definition of $r_h$ substituted into \eqref{GradientDiscreteMinRes}, it follows that \eqref{DiscreteMinResCharacterization2} $\Leftrightarrow$ \eqref{DiscreteMinResCharacterization3}.
    \end{proof}
    
%    \begin{Rem}
%        From the definition of the duality map $J_{p,\V_h}$ and taking into account that $r_h := J_{p,\V_h}^{-1} \circ I^*_h\bigl(F_h - A_h(u_h)\bigr)$, it follows that:
%        \begin{align*}
%            \norm{r_h}_{\V}^{p-1} = \norm{I^*_h(J_{p,\V(h)}(r_h))}_{(\V_h)^*} = \norm{I^*_h\bigl(F_h - A_h(u_h)\bigr)}_{(\V_h)^*} \leq \norm{F_h - A_h(u_h)}_{\V^*}.
%        \end{align*}
%    \end{Rem}

%%%%%%%%%%%%%%%%%%%%%%%%%%%%%%%%%%%%%%%%%%%%%%%%%%%%%%%%%%%%%%%%%%%%%%%%%%%%%%%%%%%%%%
%%%%%%%%%%%%%%%%%%%%%%%%%%%%%%%%%%%%%%%%%%%%%%%%%%%%%%%%%%%%%%%%%%%%%%%%%%%%%%%%%%%%%%

\subsection{Convergence analysis}
In this section, we present a convergence analysis of the discrete MinRes method applied to the $p$-Laplacian problem with $p \in (1,\infty)$, with the configuration described below.
Let $\Omega \subset \mathbb{R}^d$ (with $d \in \{2,3\}$) be a polygonal or polyhedral domain, and let $\T^h$ be a conforming simplicial mesh of $\Omega$. Let $\V_h$ be the lowest-order Crouzeix--Raviart finite element space \cite{Crouzeix1973Conforming} that vanishes at the barycenters of the boundary faces, and let $\U_h \subset \V_h$ be the H$^1$-conforming subspace. Let $\U = \V = W_0^{1,p}(\Omega)$ and let $\V(h):= \V + \V_h$ endowed with the broken norm
\begin{equation}\label{BrokenWpNorm}
    \norm{\, \bullet \,}_h^p := \sum_{T \in \T^h} |\, \bullet \,|_{W^{1,p}(T)}.
\end{equation}
As an extension of the $p$-Laplacian operator, let us consider $A_h : \U \to \V(h)^*$ defined as
\begin{align}\label{ExtendedpLaplacian}
    \DProd{A_h(u) , v}_{\V(h)^*,\V(h)} := \sum_{T \in \T^h} \int_{T} |\nabla u|^{p-2} \nabla u \cdot \nabla v\dx, \quad \forall (u, v) \in \U \times \V(h).
\end{align}

As shown below, with this configuration, the extended operator $A_h$ satisfies a property similar to one of those established in \Cref{PropertiesPLaplacian} for the $p$-Laplacian operator. The convergence  of the proposed method follows from the properties presented below.

\begin{Lem}[Local Lipschitz continuity]\label{LemmaContinuityAh}
    Let $\U = \V = W_0^{1,p}(\Omega)$ and let $\V_h$ be the Crouzeix--Raviart finite element space that vanishes at the midpoint of the boundary.
    Let $\V(h) := \V + \V_h$ endowed with the broken norm \eqref{BrokenWpNorm} and let $A_h : \U \to \V(h)^*$ be the extended $p$-Laplacian operator defined in \eqref{ExtendedpLaplacian}. Then, the operator $A_h$ satisfies the continuity properties with the same constants $C_{\mathtt{L}}$ and $c_{\mathtt{L}}$ defined in \Cref{PropertiesPLaplacian}:
    \begin{align}\label{ContinuityAh_Eq}
        \norm{A_h(u)-A_h(w)}_{\V(h)^*} \leq 
        \begin{cases}
            C_{\mathtt{L}} \Bigl(\norm{u}_{\V} + \norm{w}_{\V} \Bigr)^{p-2} \norm{u - w }_{\V} & \text{if } p \geq 2, \\
            c_{\mathtt{L}} \norm{u - w}_{\V}^{p-1} & \text{if } 1 < p < 2.
        \end{cases}
    \end{align}
\end{Lem}

\begin{proof}
    For $p=2$, the result is trivial. We proceed with the proof for $p>2$.
    We rely on the following vector inequality (see \cite[Lem. 2.1]{Liu1994Quasi}): For any $\bxi, \bzeta \in \mathbb{R}^d$, there exists a constant $C > 0$ such that
    \begin{align}\label{VectorIneqContinuity}
        \left| |\bxi|^{p-2} \bxi - |\bzeta|^{p-2} \bzeta \right| \leq C \bigl( |\bxi| + |\bzeta| \bigr)^{p-2} |\bxi - \bzeta|.
    \end{align}
    Let $u, w \in \V$. Using the definition of $A_h$ and the triangle inequality,  we have
    \begin{align*}
        \bigl|\DProd{A_h(u)-A_h(w),z}\bigr| 
        &= \left|\sum_{T \in \T^h} \int_{T} \Bigl( |\nabla u|^{p-2} \nabla u - |\nabla w|^{p-2} \nabla w \Bigr) \cdot \nabla z \dx \right| \\
        &\leq \sum_{T \in \T^h} \int_{T} \left| |\nabla u|^{p-2} \nabla u - |\nabla w|^{p-2} \nabla w \right| |\nabla z| \dx \qquad \forall z \in \V(h). 
    \end{align*}
    Applying the inequality \eqref{VectorIneqContinuity} pointwise inside the integrals, we obtain
    \begin{align*}
        \bigl|\DProd{A_h(u)-A_h(w),z}\bigr| 
        &\leq C \sum_{T \in \T^h} \int_{T} \bigl( |\nabla u| + |\nabla w| \bigr)^{p-2} |\nabla (u - w)| |\nabla z| \dx.
    \end{align*}
    Now, we apply the generalized Hölder inequality with exponents $r = \frac{p}{p-2}$, $s = p$, and $t = p$ (satisfying $1/r + 1/s + 1/t = 1$). Since the sum of integrals over disjoint elements is equivalent to the integral over the whole domain, we can apply Hölder directly to the summation
    \begin{align*}
        \bigl|\DProd{A_h(u)-A_h(w),z}\bigr| 
        &\leq C \left( \sum_{T \in \T^h} \norm{ (|\nabla u| + |\nabla w|)^{\frac{p}{r}} }_{L^{r}(T)}^{r} \right)^{\frac{1}{r}} 
       \\
        & \qquad \times  \left( \sum_{T \in \T^h} \norm{\nabla (u - w)}_{L^p(T)}^p \right)^{\frac{1}{p}}   \left( \sum_{T \in \T^h} \norm{\nabla z}_{L^p(T)}^p \right)^{\frac{1}{p}} \\
        &= C \left( \sum_{T \in \T^h} \int_T (|\nabla u| + |\nabla w|)^{p} \dx \right)^{\frac{p-2}{p}} \norm{u - w}_{\V} \norm{z}_{h}.
    \end{align*}
    Recognizing the first term as the global $L^p$-norm raised to the power $p-2$, we conclude
    \begin{align*}
        \bigl|\DProd{A_h(u)-A_h(w),z}\bigr| 
        &\leq C \norm{|\nabla u| + |\nabla w|}_{L^p(\Omega)}^{p-2} \norm{u - w}_{\V} \norm{z}_{h} \\
        &\leq C \bigl( \norm{u}_{\V} + \norm{w}_{\V} \bigr)^{p-2} \norm{u - w}_{\V} \norm{z}_{h}.
    \end{align*}
    The result follows by dividing by $\norm{z}_h$ and taking the supremum over $z \in \V(h) \setminus \{0\}$.
    
    For $1 < p < 2$, the proof follows an analogous structure. Instead of \eqref{VectorIneqContinuity}, we employ the vector inequality $\left| |\bxi|^{p-2} \bxi - |\bzeta|^{p-2} \bzeta \right| \leq 2^{2-p} |\bxi - \bzeta|^{p-1}$ valid for $p < 2$ (see, e.g., \cite[Chapter 12]{Lindqvist2017Notes}). Applying this inequality pointwise within the integrals over each element $T \in \T^h$, the result is directly obtained by using the standard Hölder's inequality with conjugate exponents $p/(p-1)$ and $p$ over the broken domain.

    The constants $C_{\mathtt{L}}$ and $c_{\mathtt{L}}$ arise exclusively from the algebraic vector inequalities applied pointwise within each element $T \in \T^h$. Consequently, they are independent of the global continuity of the functions across the mesh interfaces, and they coincide exactly with the constants of the continuous operator $A$ established in \Cref{PropertiesPLaplacian}.
\end{proof}

For the convergence analysis, we define the consistency error functional $\E_h(u)~\in~\V(h)^*$ as
\begin{equation}\label{ConsistencyErrorDef}
    \E_h(u) := F_h - A_h(u),
\end{equation}
for a given $F_h \in \V(h)^*$. We then have the following C\'ea-type estimate.

\begin{Theo}[C\'ea estimate]\label{ThmConvergenceEstimate}
    Let $u$ and $u_h$ be the solutions of the abstract problem \eqref{abstractPDE} and the discrete MinRes problem \eqref{DiscreteMinResProblem}, respectively, with the spaces, norm, and operator described in \Cref{LemmaContinuityAh}. %
    Let $\E_h(u) \in \V(h)^*$ be the consistency error functional defined in \eqref{ConsistencyErrorDef}. %
    Assume that the best approximation error is small enough such that $\inf_{w_h \in \U_h} \norm{u - w_h}_{\V} < 1$. %
    Then, the following error estimate holds:
    \[ 
    \norm{u - u_h}_{\V} \leq C_{\mathtt{app}} \inf_{w_h \in \U_h} \norm{u - w_h}_{\V}^{\alpha} + C_{\mathtt{cons}} \norm{\E_h(u)}_{(\V_h)^*}^{\beta},
    \]
    where the exponents are given by:
    \[
    \begin{aligned}
        \alpha &:=
        \begin{cases}
            \frac{1}{p-1}& \text{if } p \geq 2, \\%[1.5ex]
            p-1 & \text{if } 1 < p < 2.
        \end{cases}
            \qquad &\text{and} \qquad
        \beta &:=
        \begin{cases}
            \frac{1}{p-1}& \text{if } p \geq 2, \\%[1.5ex]
            1 & \text{if } 1 < p < 2.
        \end{cases}
    \end{aligned}
    \]
    The stability constants are defined as:
    \[
    \begin{aligned}
         C_{\mathtt{app}} &:=
         \begin{cases}
            1 + \left( \frac{2C_{\mathtt{L}} (1 + C_{\mathtt{BM}}(\V))^{p-2}}{C_{\mathtt{M}}} \right)^{\frac{1}{p-1}} \norm{F}_{\V^*}^{\frac{p-2}{(p-1)^2}} & \text{if } p \geq 2, \\[1ex]
          1 + c_{\mathtt{L}} C_{\mathtt{cons}} & \text{if } 1 < p < 2.
         \end{cases}
    \end{aligned}
    \]
    \[
    \begin{aligned}
         C_{\mathtt{cons}} &:=
         \begin{cases}
            \left(\frac{2}{C_{\mathtt{M}}}\right)^{\frac{1}{p-1}} & \text{if } p \geq 2, \\[1ex]
            \frac{2}{c_{\mathtt{M}}} \left( 2^{\frac{1}{p-1}} + C_{\mathtt{BM}}(\V) \right)^{2-p} \norm{F}_{\V^*}^{\frac{2-p}{p-1}} & \text{if } 1 < p < 2.
         \end{cases}
    \end{aligned}
    \]
\end{Theo}

\begin{proof}
    Let us start considering $p\geq 2$. Let $w_h \in \U_h$ be arbitrary. By the triangle inequality, we have
    \begin{equation}\label{TriangleIneq}
        \norm{u - u_h}_{\V} \leq \norm{u - w_h}_{\V} + \norm{u_h - w_h}_{\V}.
    \end{equation}
    We focus on bounding the second term, $\norm{u_h - w_h}_{\V}$. Since $\U_h \subset \V_h$, we can identify norms in $\V$ and $\V(h)$ restricted to $\U_h$.

    Using the strong monotonicity \eqref{MonotonicityA} and the fact that $u_h$ minimizes the residual, we derive
    \begin{align*}
        C_{\mathtt{M}} \norm{u_h - w_h}_{\V}^{p-1} &\leq \frac{\DProd{A(u_h) - A(w_h) , u_h - w_h}_{\V^*,\V}}{\norm{u_h - w_h}_{\V}} \\
        &= \frac{\DProd{A_h(u_h) - A_h(w_h) , u_h - w_h}_{\V(h)^*,\V(h)}}{\norm{u_h - w_h}_{h}} \\
        &\leq \norm{A_h(u_h) - A_h(w_h)}_{(\V_h)^*} \\
        &\leq \norm{F_h - A_h(u_h)}_{(\V_h)^*} + \norm{F_h - A_h(w_h)}_{(\V_h)^*} \\
        &\leq 2 \norm{F_h - A_h(w_h)}_{(\V_h)^*}.
    \end{align*}

    Since $A_h(u) = F_h + \E_h(u)$ in $(\V_h)^*$, we apply the triangle inequality to bound the discrete residual:
    \begin{align*}
        \norm{F_h - A_h(w_h)}_{(\V_h)^*} &\leq \norm{F_h - A_h(u)}_{(\V_h)^*} + \norm{A_h(u) - A_h(w_h)}_{(\V_h)^*} \\
        &= \norm{\E_h(u)}_{(\V_h)^*} + \norm{A_h(u) - A_h(w_h)}_{(\V_h)^*}.
    \end{align*}

    Bounding the discrete dual norm by the full dual norm and applying \Cref{LemmaContinuityAh}, we obtain
    \[ 
        C_{\mathtt{M}} \norm{u_h - w_h}_{\V}^{p-1} \leq 2 \norm{\E_h(u)}_{(\V_h)^*} + 2 C_{\mathtt{L}} \left( \norm{u}_{\V} + \norm{w_h}_{\V} \right)^{p-2} \norm{u - w_h}_{\V}.
    \]
    Thus, since $(a+b)^q \leq a^q + b^q$ for $0 < q < 1$,  
    \[ 
        \norm{u_h - w_h}_{\V} \leq \left(\frac{2}{C_{\mathtt{M}}}\right)^{\frac{1}{p-1}} \norm{\E_h(u)}_{(\V_h)^*}^{\frac{1}{p-1}} +  \left(\frac{2 C_{\mathtt{L}}}{C_{\mathtt{M}}}\right)^{\frac{1}{p-1}} \left( \norm{u}_{\V} + \norm{w_h}_{\V} \right)^{\frac{p-2}{p-1}} \norm{u - w_h}_{\V}^{\frac{1}{p-1}}.
    \]
    Substituting this back into \eqref{TriangleIneq}, taking $w_h$ as the best approximation $\hat{u}_h$ of $u$ in $\U_h$ and applying \Cref{LemStabilityBounds}, we get
    \[
    \begin{aligned}
        \norm{u - u_h}_{\V} &\leq \norm{u - \hat{u}_h}_{\V} +  \left( \frac{2C_{\mathtt{L}} (1 + C_{\mathtt{BM}}(\V))^{p-2}}{C_{\mathtt{M}}} \right)^{\frac{1}{p-1}} \norm{F}_{\V^*}^{\frac{p-2}{(p-1)^2}} \norm{u - \hat{u}_h}_{\V}^{\frac{1}{p-1}} \\
        &\hspace{7cm} + \left(\frac{2}{C_{\mathtt{M}}}\right)^{\frac{1}{p-1}} \norm{\E_h(u)}_{(\V_h)^*}^{\frac{1}{p-1}}.
    \end{aligned}
    \]
    Consequently, assuming $\norm{u - \hat{u}_h}_{\V} < 1$ (so that $\norm{u - \hat{u}_h}_{\V} \leq \norm{u - \hat{u}_h}_{\V}^{1/(p-1)}$ for $p \geq 2$), the proof is completed for $p \geq 2$.

    The proof for $1 < p < 2$ proceeds analogously. It relies on applying Lemmas \ref{PropertiesPLaplacian} and \ref{LemmaContinuityAh}, and controlling the resulting weight term $\left( \norm{u_h}_{\V} + \norm{w_h}_{\V} \right)^{2-p}$ by means of the a priori bounds provided in \Cref{LemStabilityBounds}.
\end{proof}

\begin{Rem}
    It is important to note that \Cref{ThmConvergenceEstimate} only provides a theoretical lower bound for the convergence rate. In practice, our numerical experiments exhibit convergence rates that exceed those predicted by this theorem. The suboptimal theoretical rate stems from the fact that the proof strictly relies on global vector inequalities to control the monotonicity and continuity of the operator, as in \Cref{LemmaContinuityAh}. Consequently, the analysis yields pessimistic a priori bounds because it does not exploit the local regularity of the exact solution. Such local regularity is typically essential to derive sharp and optimal error estimates for nonlinear PDEs. %\cred{Instead of mentioning in a remark, we cannot just extend Theorem 4.2 hypothesis?. Something like: Moreover, if the analytical solution satisfies ..., it holds ...}
\end{Rem}

%%%%%%%%%%%%%%%%%%%%%%%%%%%%%%%%%%%%%%%%%%%%%%%%%%%%%%%%%%%%%%%%%%%%%%%%%%%%%%%%%%%%%%
%%%%%%%%%%%%%%%%%%%%%%%%%%%%%%%%%%%%%%%%%%%%%%%%%%%%%%%%%%%%%%%%%%%%%%%%%%%%%%%%%%%%%%
\subsection{Discrete a posteriori error estimator}

In this section, we derive an a posteriori error estimator for the discrete MinRes method. We define the discrete residual representative $r_h \in \V_h$ as in \Cref{DiscreteMinResCharacterization}.

First, we establish the estimator's efficiency (lower bound). This result relies solely on the properties of the operator $A_h$ and the approximation properties of the space $\U_h$, without requiring additional stability assumptions.

\begin{Prop}[Efficiency]\label{PropEfficiency}
    Let $u$ and $u_h$ be the solutions of problems \eqref{abstractPDE} and \eqref{DiscreteMinResProblem}, respectively.
    Let $\E_h(u) \in \V(h)^*$ be the consistency error functional defined in \eqref{ConsistencyErrorDef}.
    Then, the following lower bounds hold:
    \begin{equation}\label{DiscreteErrorEstimate1}
        \begin{cases}
            \norm{r_h}_{h}^{p-1} \leq C \norm{u - u_h}_{\V} + \norm{\E_h(u)}_{(\V_h)^*} & \text{if } p \geq 2, \\[1.5ex]
            \norm{r_h}_{h}^{p-1} \leq c_{\mathtt{L}} \norm{u - u_h}_{\V}^{p-1} + \norm{\E_h(u)}_{(\V_h)^*}  & \text{if } 1 < p < 2,
        \end{cases}
    \end{equation}
    where $C := C_{\mathtt{L}} (1 + C_{\mathtt{BM}}(\V))^{p-2} \norm{F}_{\V^*}^{\frac{p-2}{p-1}}$.
\end{Prop}

\begin{proof}
    Recall that $\norm{r_h}_{h}^{p-1} = \norm{F_h - A_h(u_h)}_{(\V_h)^*}$.
    Since $u_h$ is the minimizer in $\U_h$, for any $w_h \in \U_h$, we have
    \begin{align*}
        \norm{r_h}_{h}^{p-1} &\leq \norm{F_h - A_h(w_h)}_{(\V_h)^*} \\
        &\leq \norm{A_h(u) - A_h(w_h)}_{(\V_h)^*} + \norm{F_h - A_h(u)}_{(\V_h)^*} \\
        &= \norm{A_h(u) - A_h(w_h)}_{(\V_h)^*} + \norm{\E_h(u)}_{(\V_h)^*}.
    \end{align*}
    
    \noindent\textbf{For $p \geq 2$:} Applying \Cref{LemmaContinuityAh} for the specific approximation $w_h := \hat{u}_h$ defined in \Cref{LemStabilityBounds}, and invoking the a priori bounds provided therein, gives
    \[
        \norm{r_h}_{h}^{p-1} \leq C_{\mathtt{L}} (1 + C_{\mathtt{BM}}(\V))^{p-2} \norm{F}_{\V^*}^{\frac{p-2}{p-1}} \norm{u - \hat{u}_h}_{\V} + \norm{\E_h(u)}_{(\V_h)^*}.
    \]
    The result follows directly since $\norm{u - \hat{u}_h}_{\V} \leq \norm{u - u_h}_{\V}$.
    
    \smallskip
    \noindent\textbf{For $1 < p < 2$:} We evaluate the error bound directly by choosing $w_h := u_h$. Applying the Hölder continuity from \Cref{LemmaContinuityAh}, it immediately yields:
    \begin{align*}
        \norm{r_h}_{h}^{p-1} &\leq \norm{A_h(u) - A_h(u_h)}_{(\V_h)^*} + \norm{\E_h(u)}_{(\V_h)^*} \\
        &\leq c_{\mathtt{L}} \norm{u - u_h}_{\V}^{p-1} + \norm{\E_h(u)}_{(\V_h)^*},
    \end{align*}
    which completes the proof.
\end{proof}

On the other hand, establishing the estimator's reliability (upper bound) requires the following additional assumption.

\begin{Assum}[Fortin operator]\label{AssumFortin}
    There exists a linear operator $\Pi_h : \V(h) \to \V_h$ and a constant $C_{\Pi}>0$ such that:
    \begin{align*}
    \begin{cases}
        \norm{\Pi_h v}_{h} \leq C_{\Pi} \norm{v}_{h}, \quad &\forall v \in \V(h), \\
        \DProd{A_h(w_h) , v - \Pi_h v}_{\V(h)^*,\V(h)} = 0, \quad &\forall w_h \in \U_h, \ \forall v \in \V(h).
    \end{cases}
    \end{align*}
\end{Assum}

Under this assumption, we define the data oscillation term as:
\begin{align*}
    \osc(F_h) := \sup_{v \in \V(h)} \dfrac{\DProd{F_h , v - \Pi_{h} v}_{\V(h)^*,\V(h)}}{\norm{v}_{h}}.
\end{align*}

We are now in a position to state the reliability estimate.

\begin{Theo}[Reliability]\label{ThmReliability}
    Let $u$ and $u_h$ be the solutions of problems \eqref{abstractPDE} and \eqref{DiscreteMinResProblem}, respectively.
    Let $\E_h(u) \in \V(h)^*$ be the consistency error functional, defined such that $F_h - A_h(u) = \E_h(u)$ in $\V(h)^*$.
    Under \Cref{AssumFortin}, the following upper bound holds:
    \begin{align}\label{DiscreteErrorEstimate2}
        \norm{u - u_h}_{\V} \lesssim 
        \begin{cases}
            \norm{\E_h(u)}_{\V(h)^*}^{\frac{1}{p-1}} + \osc(F_h)^{\frac{1}{p-1}} + \norm{r_h}_{h} & \text{if } p \geq 2, \\[1.5ex]
            \norm{\E_h(u)}_{\V(h)^*} + \osc(F_h) + \norm{r_h}_{h}^{p-1} & \text{if } 1 < p < 2.
        \end{cases}
    \end{align}
    Here, the hidden constants depend on the monotonicity constants, $C_{\Pi}$, $p$, and, for $1 < p < 2$, on the a priori bounds of $u$ and $u_h$.
\end{Theo}

\begin{proof}
    For any $v \in \V(h)$, we decompose the action of the error using the identity $A_h(u) + \E_h(u) = F_h$ in $\V(h)^*$ and the splitting $v = (v - \Pi_h v) + \Pi_h v$:
    \begin{align*}
        \DProd{A_h(u) - A_h(u_h), v}_{\V(h)^*,\V(h)} 
        &= \DProd{A_h(u) - F_h, v}_{\V(h)^*,\V(h)} + \DProd{F_h - A_h(u_h), v}_{\V(h)^*,\V(h)} \\
        &= -\DProd{\E_h(u), v}_{\V(h)^*,\V(h)} + \DProd{F_h, v - \Pi_h v}_{\V(h)^*,\V(h)} \\ 
        &\quad - \DProd{A_h(u_h), v - \Pi_h v}_{\V(h)^*,\V(h)} \\
        &\quad + \DProd{F_h - A_h(u_h), \Pi_h v}_{\V(h)^*,\V(h)}.
    \end{align*}

    Using the orthogonality property of the Fortin operator $\Pi_h$ (since $u_h \in \U_h$), considering that $J_{p,\V(h)}(r_h) = F_h - A_h(u_h)$ in $(\V_h)^*$, we derive:
    \begin{align*}
        \DProd{A_h(u) - A_h(u_h), v}_{\V(h)^*,\V(h)} 
        &\leq \norm{\E_h(u)}_{\V(h)^*} \norm{v}_h + \osc(F_h) \norm{v}_h \\
        &\quad + \DProd{J_{p,\V(h)}(r_h), \Pi_h v}_{\V(h)^*,\V(h)} \\
        &\leq \left( \norm{\E_h(u)}_{\V(h)^*} + \osc(F_h) + C_{\Pi} \norm{r_h}_{h}^{p-1} \right) \norm{v}_h.
    \end{align*}

    Dividing by $\norm{v}_h$ and taking the supremum over $v \in \V(h)$, we obtain the bound for the dual norm:
    \begin{align}\label{DualNormBound}
        \norm{A_h(u) - A_h(u_h)}_{\V(h)^*} \leq \norm{\E_h(u)}_{\V(h)^*} + \osc(F_h) + C_{\Pi} \norm{r_h}_{h}^{p-1}.
    \end{align}

    \noindent\textbf{For $p \geq 2$:} Using the strong monotonicity \eqref{MonotonicityA}, we have $C_{\mathtt{M}} \norm{u - u_h}_{\V}^{p-1} \leq \norm{A_h(u) - A_h(u_h)}_{\V(h)^*}$. Combining this with \eqref{DualNormBound}, raising both sides to the power $\frac{1}{p-1}$, and applying the inequality $(a+b+c)^q \leq a^q + b^q + c^q$ (valid for $0 < q \leq 1$), we obtain the desired result.

    \smallskip 
    \noindent\textbf{For $1 < p < 2$:} Using the degenerate strong monotonicity \eqref{MonotonicityA}, we derive
        \begin{align*}
            c_{\mathtt{M}} \left( \norm{u}_{h} + \norm{u_h}_{h} \right)^{p-2} \norm{u - u_h}_{\V} \leq \norm{A_h(u) - A_h(u_h)}_{\V(h)^*}.
        \end{align*}
        Utilizing \eqref{DualNormBound} and bounding the weight term $\left( \norm{u}_{h} + \norm{u_h}_{h} \right)^{2-p}$ with the a priori bounds from \Cref{LemStabilityBounds}, the upper estimate follows.
\end{proof}

The reliability estimate in \Cref{ThmReliability} involves the data oscillation term $\osc(F_h)$, which depends on the smoothness of the data relative to the discrete space. The following proposition establishes that $\osc(F_h)$ is indeed bounded by the best approximation error of the exact solution.

\begin{Prop}[Bound on data oscillation]\label{PropOscillationBound}
    Let $u$ be the solution of \eqref{abstractPDE}. Let $\E_h(u) \in \V(h)^*$ be the consistency error functional defined in \eqref{ConsistencyErrorDef}. Under \Cref{AssumFortin}, for all $p > 1$, the data oscillation term satisfies:
    \begin{equation}\label{OscillationBound}
        \osc(F_h) \leq C \inf_{w_h \in \U_h} \norm{u - w_h}_{\V}^{\alpha} + (1+C_{\Pi}) \norm{\E_h(u)}_{(\V_h)^*},
    \end{equation}
    where $\alpha = \min\{1,p-1\}$ and
    \[
        C = 
        \begin{cases}
            \displaystyle C_{\mathtt{L}} (1+C_{\Pi}) (1 + C_{\mathtt{BM}}(\V))^{p-2} \norm{F}_{\V^*}^{\frac{p-2}{p-1}} & \text{if } p \geq 2, \\[1.5ex]
            \displaystyle c_{\mathtt{L}} (1+C_{\Pi}) & \text{if } 1 < p < 2.
        \end{cases}
    \]
\end{Prop}

\begin{proof}
    Let $v \in \V(h)$ and let $w_h \in \U_h$ be arbitrary. Using the algebraic identity $F_h = A_h(u) + \E_h(u)$ in $\V(h)^*$, and the orthogonality property of the Fortin operator $\DProd{A_h(w_h) , v - \Pi_h v} = 0$ (\Cref{AssumFortin}), we derive
    \begin{align*}
        \DProd{F_h , v - \Pi_h v}_{\V(h)^*,\V(h)} &= \DProd{A_h(u) + \E_h(u) - A_h(w_h) , v - \Pi_h v}_{\V(h)^*,\V(h)} \\
        &\leq \left( \norm{A_h(u) - A_h(w_h)}_{\V(h)^*} + \norm{\E_h(u)}_{\V(h)^*} \right) \norm{v - \Pi_h v}_{h} \\
        &\leq (1+C_{\Pi}) \left( \norm{A_h(u) - A_h(w_h)}_{\V(h)^*} + \norm{\E_h(u)}_{\V(h)^*} \right) \norm{v}_{h}.
    \end{align*}
    Dividing by $\norm{v}_h$ and taking the supremum over $v \in \V(h) \setminus \{0\}$, it follows that
    \begin{align}\label{osc_bound}
        \osc(F_h) &\leq (1 + C_{\Pi}) \norm{A_h(u) - A_h(w_h)}_{\V(h)^*} + (1+C_{\Pi}) \norm{\E_h(u)}_{\V(h)^*}.
    \end{align}

    \noindent\textbf{For $p \geq 2$:} We apply the Lipschitz continuity of $A_h$ from \Cref{LemmaContinuityAh} to bound the first term. By choosing $w_h$ as the best approximation $\hat{u}_h$ and using the a priori bounds from \Cref{LemStabilityBounds} to control the weight $\left( \norm{u}_{\V} + \norm{\hat{u}_h}_{\V} \right)^{p-2}$, we obtain the constant $C$ and the linear dependence on $\norm{u - \hat{u}_h}_{\V}$.
    
    \smallskip
    \noindent\textbf{For $1 < p < 2$:} We apply the Hölder continuity of $A_h$ from \Cref{LemmaContinuityAh} directly, which immediately yields the bound in \eqref{osc_bound} with exponent $p-1$, completing the proof.
\end{proof}

%%%%%%%%%%%%%%%%%%%%%%%%%%%%%%%%%%%%%%%%%%%%%%%%%%%%%%%%%%%%%%%%%%%%%%%%%%%%%%%%%%%%%%
%%%%%%%%%%%%%%%%%%%%%%%%%%%%%%%%%%%%%%%%%%%%%%%%%%%%%%%%%%%%%%%%%%%%%%%%%%%%%%%%%%%%%%
%%%%%%%%%%%%%%%%%%%%%%%%%%%%%%%%%%%%%%%%%%%%%%%%%%%%%%%%%%%%%%%%%%%%%%%%%%%%%%%%%%%%%%
%%%%%%%%%%%%%%%%%%%%%%%%%%%%%%%%%%%%%%%%%%%%%%%%%%%%%%%%%%%%%%%%%%%%%%%%%%%%%%%%%%%%%%
%%%%%%%%%%%%%%%%%%%%%%%%%%%%%%%%%%%%%%%%%%%%%%%%%%%%%%%%%%%%%%%%%%%%%%%%%%%%%%%%%%%%%%
%%%%%%%%%%%%%%%%%%%%%%%%%%%%%%%%%%%%%%%%%%%%%%%%%%%%%%%%%%%%%%%%%%%%%%%%%%%%%%%%%%%%%%
%%%%%%%%%%%%%%%%%%%%%%%%%%%%%%%%%%%%%%%%%%%%%%%%%%%%%%%%%%%%%%%%%%%%%%%%%%%%%%%%%%%%%%
%%%%%%%%%%%%%%%%%%%%%%%%%%%%%%%%%%%%%%%%%%%%%%%%%%%%%%%%%%%%%%%%%%%%%%%%%%%%%%%%%%%%%%
%%%%%%%%%%%%%%%%%%%%%%%%%%%%%%%%%%%%%%%%%%%%%%%%%%%%%%%%%%%%%%%%%%%%%%%%%%%%%%%%%%%%%%
%%%%%%%%%%%%%%%%%%%%%%%%%%%%%%%%%%%%%%%%%%%%%%%%%%%%%%%%%%%%%%%%%%%%%%%%%%%%%%%%%%%%%%
%%%%%%%%%%%%%%%%%%%%%%%%%%%%%%%%%%%%%%%%%%%%%%%%%%%%%%%%%%%%%%%%%%%%%%%%%%%%%%%%%%%%%%
%%%%%%%%%%%%%%%%%%%%%%%%%%%%%%%%%%%%%%%%%%%%%%%%%%%%%%%%%%%%%%%%%%%%%%%%%%%%%%%%%%%%%%
%%%%%%%%%%%%%%%%%%%%%%%%%%%%%%%%%%%%%%%%%%%%%%%%%%%%%%%%%%%%%%%%%%%%%%%%%%%%%%%%%%%%%%
%%%%%%%%%%%%%%%%%%%%%%%%%%%%%%%%%%%%%%%%%%%%%%%%%%%%%%%%%%%%%%%%%%%%%%%%%%%%%%%%%%%%%%

\section{Numerical implementation}\label{sec:NumericalImplementation}

In this section, we present results obtained by applying the mixed formulation in \Cref{DiscreteMinResCharacterization} to a nonlinear PDE using a finite element discretization. The numerical realization of the algorithms has been carried out using the open-source library \texttt{Netgen/NGSolve} \url{https://www.ngsolve.org}.  

\subsection{Discretization of the model problem}

    Let $\Omega^h \subset \Omega$ be a polygonal {(or polyhedral)} approximation to $\Omega$. {Let $\T^h$ be a partition of $\Omega^h$ into a finite number of disjoint open regular simplices $T$ (triangles for $d=2$ or tetrahedra for $d=3$) such that $\bigcup_{T \in \T^h} \overline{T} = \overline{\Omega^h}$, where $h:= \max_{T \in \T^h} h_T$ with $h_T$ being the maximum diameter of $T$. Suppose that each vertex associated with the simplicial mesh $\T^h$ that is in $\partial \Omega^h$ is also in $\partial \Omega$. Assume further that each element has at most one $(d-1)$-dimensional face on $\partial \Omega^h$ and that two elements $T$ and $T'$ such that $\overline{T} \cap \overline{T'} \neq \emptyset$ share either a common vertex, a common edge, or a common whole $(d-1)$-dimensional face.}
    The following two finite element spaces are defined on the partition described above (for the foundational formulation and its 3D extensions, we refer to \cite{Crouzeix1973Conforming, Ciarlet2018Family}):
    \begin{align*}
        \U_h &:= \left\{ u_h \in W_0^{1,p}(\Omega^h): u_h|_T \in \P_1(T), \forall T \in \T^h \right\}, \\
        %\V_h &:= \Bigl\{ v_h \in L^p(\Omega^h): v_h|_T \in \P_1(T),\ \forall T \in \T^h, \\
        %&\hspace{0.8cm} v_h \text{ is continuous at the midpoints of the edges of } \T^h \\
        %&\hspace{1.2cm} \text{and } v_h = 0 \text{ at the midpoints of the edges on } \partial \Omega^h \Bigr\},
        \V_h &:= \Bigl\{ v_h \in L^p(\Omega^h): v_h|_T \in \P_1(T),\ \forall T \in \T^h, v_h \text{ is continuous at the interior} \\ 
        &\hspace{0.8cm} \text{face barycenters of } 
         \T^h, \text{ and vanishes at the face barycenters of } \partial \Omega^h \Bigr\},
    \end{align*}
    where $\P_1(T)$ is the space of polynomials of degree up to $1$ defined on the element $T$. The finite element spaces $\U_h$ and $\V_h$ are   endowed with the norms
    \begin{equation}
        |\, \bullet \,|^p_{\U_h} := \sum_{i=1}^d \norm{\frac{\partial(\, \bullet \,)}{\partial x_i}}^p_{L^p(\Omega^h)}, \label{norm_U_h} \qquad 
        |\, \bullet \,|^p_{\V_h} := \sum_{ T \in \T^h} \sum_{i=1}^d \norm{\frac{\partial(\, \bullet \,)}{\partial x_i}}^p_{L^p(T)}, %\nonumber
    \end{equation}
    respectively. Note that $\U_h$ is a subspace of $W_0^{1,p}(\Omega^h)$, but $\V_h \not\subset W_0^{1,p}(\Omega^h)$. Furthermore, it follows directly that $\U_h \subset \V_h$.

%\textcolor{orange}{[we should mention  somewhere what happens with higher polynomial degrees (not just CR1)]}

    The pair of spaces $(\U_h, \V_h)$ defined above satisfies the stability condition necessary for the well-posedness of the discrete method.
    
    \begin{Prop}[Verification of Fortin condition]\label{Prop:FortinVerification}
        Let $\V(h) = \V + \V_h$ endowed with the broken norm $\norm{\, \bullet \,}_h$. There exists a Fortin operator $\Pi_h : \V(h) \to \V_h$ satisfying  \Cref{AssumFortin}, i.e.,
        \begin{align*}
        \begin{cases}
            \norm{\Pi_h v}_{h} \leq C_{\Pi} \norm{v}_{h}, \quad &\forall v \in \V(h), \\
            \DProd{A_h(w_h) , v - \Pi_h v}_{\V(h)^*,\V(h)} = 0, \quad &\forall w_h \in \U_h, \ \forall v \in \V(h).
        \end{cases}
        \end{align*}
    \end{Prop}

    \begin{proof}
        Let us define $\Pi_h : \V(h) \to \V_h$ as the Crouzeix--Raviart interpolation operator. For any $v \in \V(h)$, its projection $\Pi_h v$ is defined element-wise such that, for all $T \in \T^h$, $\Pi_h v|_T \in \mathbb{P}_1(T)$ is the unique linear polynomial satisfying:
        \begin{align}\label{CR_Definition}
            \int_{F} \Pi_h v \ds = \int_{F} v \ds, \quad \text{for all } (d-1)\text{-dimensional faces } F \subset \partial T.
        \end{align}
        Note that this operator is well-defined on $\V(h)$ since functions in $\V$ have traces on the faces, and functions in $\V_h$ are continuous at the barycenters of the faces, ensuring the integrals are single-valued. Moreover, classical interpolation theory ensures that $\Pi_h$ is bounded in the broken norm \cite[Th. 4.4.20]{Brenner2008Mathematical}.
        
        To prove the orthogonality condition, we first analyze the gradient of the projected function. Using the divergence theorem on an element $T$, for any smooth function $\phi$, we have $\int_T \nabla \phi \dx = \int_{\partial T} \phi \mathbf{n}_T \ds$. Since $\Pi_h v$ is linear, $\nabla (\Pi_h v)$ is constant on $T$, and $\mathbf{n}_T$ is piecewise constant on the faces $F \subset \partial T$. Thus:
        \begin{align*}
            \int_T \nabla (\Pi_h v) \dx &= \sum_{F \subset \partial T} \mathbf{n}_F \int_{F} \Pi_h v \ds \\
            &= \sum_{F \subset \partial T} \mathbf{n}_F \int_{F} v \ds \quad \text{(by definition \eqref{CR_Definition})} \\
            &= \int_{\partial T} v \mathbf{n}_T \ds = \int_T \nabla v \dx.
        \end{align*}
        This implies that $\int_T \nabla(v - \Pi_h v) \dx = 0$ for all $T \in \T^h$.

        Finally, considering the discrete operator $A_h$ (the extended $p$-Laplacian), and noting that for any $w_h \in \U_h$, the gradient $\nabla w_h|_T$ is a constant vector (since $w_h \in \mathbb{P}_1(T)$), the term $|\nabla w_h|^{p-2} \nabla w_h$ is also constant on each element. Therefore,
        \begin{align*}
            \DProd{A_h(w_h) , v - \Pi_h v}_{\V(h)^*,\V(h)} &= \sum_{T \in \T^h} \int_T |\nabla w_h|^{p-2} \nabla w_h \cdot \nabla (v - \Pi_h v) \dx \\
            &= \sum_{T \in \T^h} |\nabla w_h|^{p-2} \nabla w_h \cdot \int_T \nabla (v - \Pi_h v) \dx = 0.
        \end{align*}
        This confirms that $\Pi_h$ satisfies the Fortin condition.
    \end{proof}
    
    The model problem we consider is the $p$-Laplacian equation with homogeneous Dirichlet boundary conditions described in \eqref{pLaplacian_problem}. To simplify the problem, assume that $\Omega^h = \Omega$ and $f \in L^{p^*}(\Omega)$. The numerical approximation in $\U_h$ is sought through the mixed framework established in \Cref{DiscreteMinResCharacterization}, featuring the discrete operator $A_h$ from \eqref{ExtendedpLaplacian} and the source functional $F_h \in \V(h)^*$ defined as follows:
    \[
    \begin{aligned} 
        \DProd{F_h , v_h}_{\V(h)^*,\V(h)} &:= \int_{\Omega} f v_h \dx.
    \end{aligned}
    \]

%%%%%%%%%%%%%%%%%%%%%%%%%%%%%%%%%%%%%%%%%%%%%%%%%%%%%%%%%%%%%%%%%%%%%%%%%%%%%%%%%%%%%%
%%%%%%%%%%%%%%%%%%%%%%%%%%%%%%%%%%%%%%%%%%%%%%%%%%%%%%%%%%%%%%%%%%%%%%%%%%%%%%%%%%%%%%

\subsection{The proposed algorithm}

To implement the theoretical results from the preceding sections, we have developed a finite element algorithm derived from the mixed formulation in 
\Cref{DiscreteMinResCharacterization}. Since this formulation involves the duality map (which is a nonlinear operator) and $A_h$ is also generally nonlinear, our algorithm employs a damped Newton's method to solve the nonlinear system arising from the mixed formulation in \Cref{DiscreteMinResCharacterization} as follows: Given the discrete solution pair $(r_h^j, u_h^j) \in \V_h \times \U_h$ in an iterative step $j$, we seek for the increment $(\delta r_h,\delta u_h)$ in the next iteration step,  solving the following linearized problem at the iteration $j + 1$:
\small
\begin{align}\label{NewtonStep}
    \begin{cases}
        \nabla J_{s,\V_h}(r_h^j; \delta r_h,v_h) + \nabla A_h(u_h^j; \delta u_h , v_h) &= F_h(v_h) - J_{s,\V_h}(r_h^j;v_h) - A_h(u_h^j;v_h), \\
        \nabla A_h(u_h^j;w_h,\delta r_h) &= - \nabla A_h(u_h^j;w_h,r_h^j),
    \end{cases}
\end{align}
\normalsize
for all $(v_h,w_h) \in \V_h \times \U_h$. Then, update $r_h^{j+1} := r_h^j + \delta r_h$ and $u_h^{j+1} := u_h^j + \delta u_h$. Note that, for a more compact presentation of the iterative scheme \eqref{NewtonStep}, we have avoided the use of notation \eqref{DualityPairing}.

It is well known that the convergence of Newton's method is sensitive to the initial guess, particularly for values of $p$ far from $2$ where the nonlinearity becomes severe. For this reason, we employ a parameter continuation strategy on the exponent $p$. The process begins by solving the problem for $p=2$ (the linear case), which guarantees a unique solution and requires no initial guess. This solution serves as the initial guess for a slightly perturbed problem. We proceed iteratively, updating the exponent $p$ by a small step $\delta p$ (denoted as \texttt{step} in \Cref{AlgMinRes}). Specifically, the update direction depends on the target exponent: we increment $p$ if the target is greater than $2$, and decrement it if the target is less than $2$. This logic is represented by the notation $2 \pm \texttt{step}$ in the algorithm's loop. We continue this process, using the solution from the previous step as the initial guess for the current Newton solver, until the target value of $p$ is reached. This procedure is outlined in \Cref{AlgMinRes}. 

Other numerical tests (not shown here) have been carried out using strong imposition of Dirichlet boundary conditions and standard full Newton. For singular loads we initially observed that the estimator tended to underestimate the true error decay for $p < 2$ and overestimate it for $p > 2$. On the other hand, using Nitsche's method, the computations were
highly sensitive to its penalty parameter, typically requiring huge values.
In light of these challenges, we opted to retain the strong Dirichlet formulation and explore mesh-based strategies to recover the estimator's efficiency, as will be thoroughly examined in the numerical tests below.

\begin{algorithm}[t!]
\caption{MinRes continuation algorithm for the $p$-Laplacian.}
{\small \begin{algorithmic}[1]\label{AlgMinRes}
\STATE input $p$, \rm{step}:=0.10, MaxNewton, MaxRef, tol.
\FOR{$i \in \{1, \cdots, {\rm MaxRef}\}$ }
    \STATE Define test and trial spaces.
    \STATE Compute the number of DOFs.
    \STATE Compute $(r_h^0,u_h^0)$ for $P = 2$ (linear case).
    \FOR{$P = 2 \pm {\rm step}$ \TO $p$}
        \WHILE{$\norm{(\delta r_h, \delta u_h)} < \text{tol}$} \label{NewtonMethod}
            \STATE Solve \eqref{NewtonStep} to compute $(\delta r_h, \delta u_h)$.
            \STATE Update $(r_h^{j+1},u_h^{j+1}) \leftarrow (r_h^j + \delta r_h,u_h^j + \delta u_h)$.
        \ENDWHILE
        \IF{MaxNewton is reached}
            \STATE Refine the step and go to \ref{NewtonMethod}.
        \ENDIF
    \ENDFOR
    \STATE Compute the error estimate and the actual error.
    \STATE Refine the mesh uniformly.
\ENDFOR
\end{algorithmic}}
\end{algorithm}

%%%%%%%%%%%%%%%%%%%%%%%%%%%%%%%%%%%%%%%%%%%%%%%%%%%%%%%%%%%%%%%%%%%%%%%%%%%%%%%%%%%%%%
%%%%%%%%%%%%%%%%%%%%%%%%%%%%%%%%%%%%%%%%%%%%%%%%%%%%%%%%%%%%%%%%%%%%%%%%%%%%%%%%%%%%%%

\subsection{Numerical results}

In this section, we present numerical experiments to validate the performance of the proposed discrete MinRes method and the associated a posteriori error estimator. {We consider the model problem \eqref{pLaplacian_problem} on the domain $\Omega = (0,1)^d$ with $d \in \{2,3\}$. Following the benchmark problems proposed in \cite{Liu2001Some}, the exact solution is prescribed as the radially symmetric function:
\begin{align}\label{AnalyticalSolution}
    u(\boldsymbol{x}) = \frac{p-1}{p-\sigma} \left( \frac{1}{d - \sigma} \right)^{\frac{1}{p-1}} \left(1 - r^{\frac{p-\sigma}{p-1}} \right),
\end{align}
where $r = |\boldsymbol{x} - \boldsymbol{x}_0|$ and $\sigma < d$. The source term is computed analytically as $f(r) = r^{-\sigma}$. In all experiments, we set $\sigma = 0.97$. The domain is discretized using a sequence of uniform simplicial meshes.}

We report the error measured in the norm \eqref{norm_U_h} and the computed estimator $\eta := \|r_h\|_{\V_h}^{p-1}$, which corresponds to the discrete dual norm of the residual (see \Cref{DiscreteMinResCharacterization}). According to our theoretical findings, this quantity is expected to scale with the error. In all our experiments, unless otherwise stated, the initial mesh is generated by a standard diagonal split of the unit square followed by a single uniform refinement step. 
{We consider four test cases to analyze the  behavior of the proposed method across different regimes of $p$ and solution regularity.}

\smallskip 
\noindent\textbf{Case 1: Smooth solutions in 2D and 3D.}
To evaluate the performance of the method under high regularity conditions, we set $\boldsymbol{x}_0 = (-1, -1)^T$ for the two-dimensional domain and $\boldsymbol{x}_0 = (-1, -1, -1)^T$ for the three-dimensional domain, placing the singularity of the radial function outside the computational domain. With this configuration, the exact solution is highly regular, $u \in C^{\infty}(\overline{\Omega})$. We test both the singular ($p=1.5$) and degenerate ($p=3.0$) regimes under uniform mesh refinement.

\Cref{Fig:SmoothCases} illustrates the convergence history under uniform mesh refinement for all four configurations. In both dimensions and for both exponents, we observe that the actual error $\|u - u_h\|_{\U_h}$ and the computed estimator $\eta$ converge at the optimal rate of $\mathcal{O}(h)$ (which corresponds to $\mathcal{O}(\text{NDOFs}^{-1/2})$ in 2D and $\mathcal{O}(\text{NDOFs}^{-1/3})$ in 3D) for linear elements. The estimator closely tracks the true error from the earliest refinement steps without the need for artificial stabilization, validating the robustness of the dual-norm-based estimator for both $p < 2$ and $p > 2$ under strong boundary conditions.

In addition, to evaluate the computational cost and scalability of our numerical approach, we track the total number of Newton iterations accumulated over the entire parameter continuation path. \Cref{Tab:NewtonIters_Smooth_Combined} summarizes these counts. The nonlinear solver demonstrates remarkable robustness; notably, in the degenerate regime ($p=3$), the number of iterations remains practically constant despite significant increases in the degrees of freedom across both dimensions.

\begin{figure}[t!]
    \centering
    % Top row: 2D results
    \includegraphics[width=0.36\textwidth]{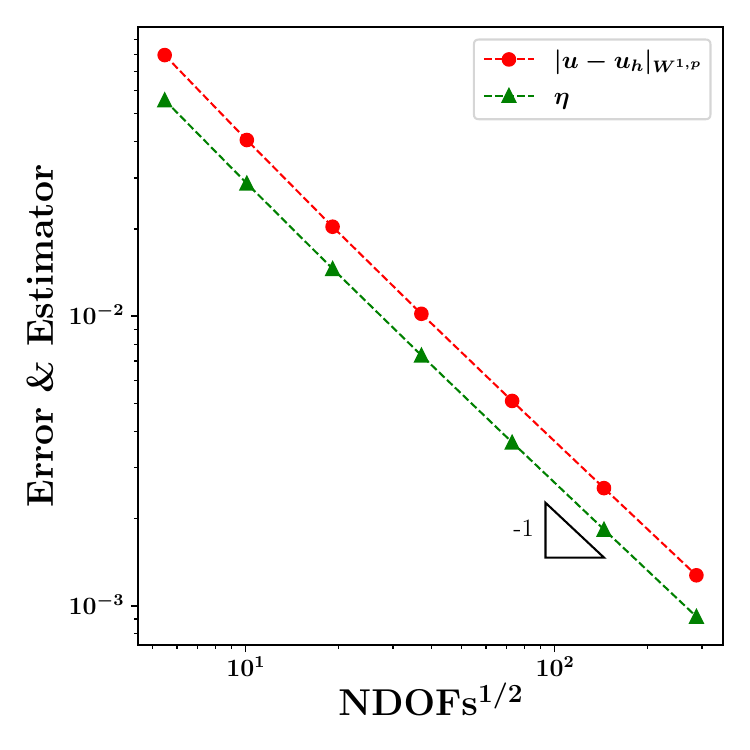}
    \includegraphics[width=0.36\textwidth]{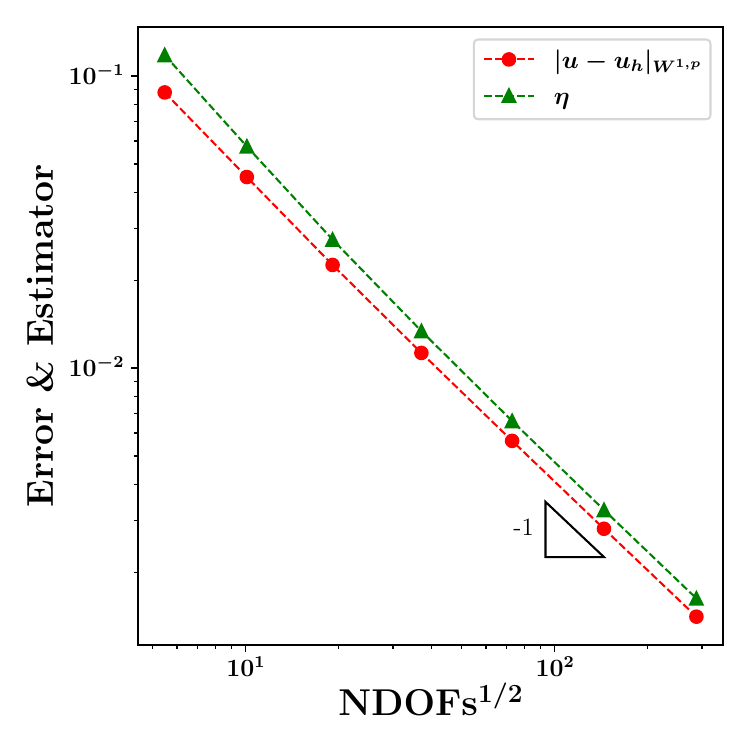} \\[-1ex] % Small vertical adjustment
    % Bottom row: 3D results
    \includegraphics[width=0.36\textwidth]{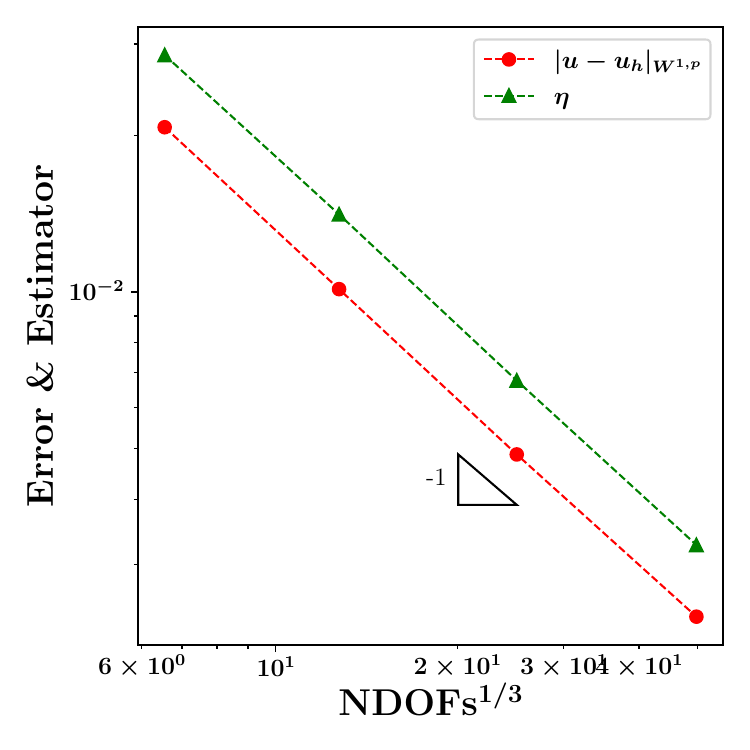}
    \includegraphics[width=0.36\textwidth]{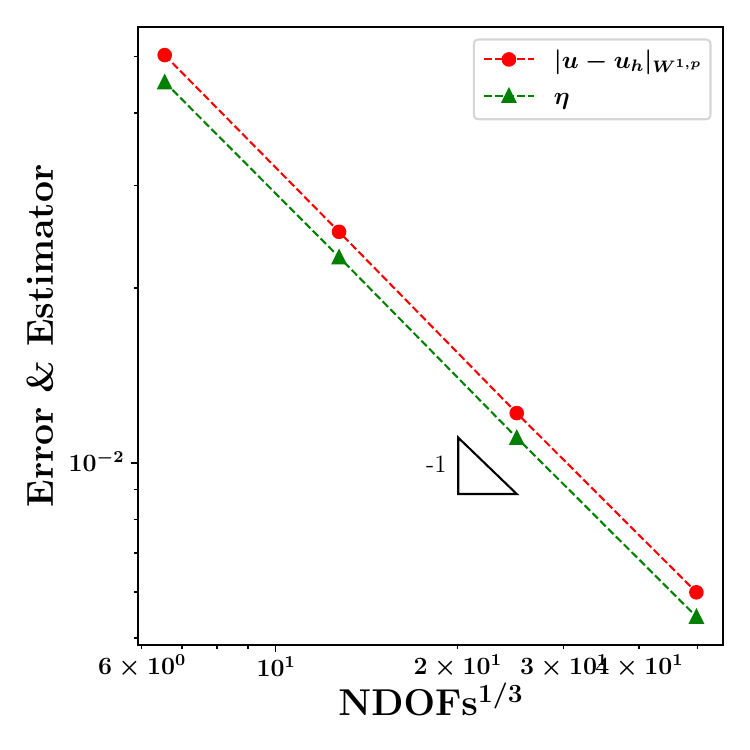}
    \vspace{-2mm}
    \caption{Convergence rates for smooth solutions. Top row: 2D results for $p=1.5$ (left) and $p=3$ (right). Bottom row: 3D results for $p=1.5$ (left) and $p=3$ (right).}
    \label{Fig:SmoothCases}
\end{figure}

\begin{table}[t!]
    \centering
    \small
    \begin{tabular}{c | c c | c | c c}
        \toprule
        \multicolumn{3}{c|}{\textbf{2D Smooth Cases}} & \multicolumn{3}{c}{\textbf{3D Smooth Cases}} \\
        \cmidrule(r){1-3} \cmidrule(l){4-6}
        \textbf{NDOFs} & \textbf{$p=1.5$} & \textbf{$p=3.0$} & \textbf{NDOFs} & \textbf{$p=1.5$} & \textbf{$p=3.0$} \\
        \midrule
        30    & 26  & 55 & 281    & 44  & 57 \\
        102   & 35  & 56 & 2067   & 56  & 58 \\
        366   & 39  & 49 & 15791  & 70  & 60 \\
        1374  & 44  & 48 & 123399 & 117 & 60 \\
        5310  & 63  & 47 & -      & -   & -  \\
        20862 & 79  & 44 & -      & -   & -  \\
        82686 & 159 & 40 & -      & -   & -  \\
        \bottomrule
    \end{tabular}
    \caption{Total accumulated Newton iterations required to reach the target exponent $p$ during the continuation process for smooth solutions under uniform refinement.}
    \label{Tab:NewtonIters_Smooth_Combined}
\end{table}

\smallskip 
\noindent\textbf{Case 2: Singular right-hand side with $p=1.5$.}
We now consider a problem with lower regularity by setting the source term singularity at the corner of the domain ($\boldsymbol{x}_0 = (0, 0)^T$). We set the exponent to $p=1.5$. To fully understand the behavior of our method under these conditions, and specifically how the resolution of the right-hand side affects the error estimator, we analyze three mesh-refinement strategies.

First, we apply the standard uniform refinement starting from the previously described coarse initial mesh. The results are shown in \Cref{Fig:Case2_Rates} (left). In this scenario, the estimator's convergence rate underestimates the true error rate. We attribute this discrepancy to the right-hand side singularity, which is not adequately represented by the standard initial mesh and remains poorly resolved under purely uniform refinement. %\textcolor{orange}{[This is essentially the same comment as at the end of Sect 5.2. Maybe we can drop it from here or from there]}

To verify our hypothesis regarding the impact of the right-hand-side representation, we introduce an initial mesh-refinement  inspired by the pre-adaptation strategies from  \cite{Millar2022Projection}. While that reference proposes starting   with a regularized right-hand side, we leverage only the mesh-adaptation aspect of their procedure to isolate its geometric benefits. This provides us with a pre-adapted initial mesh that more accurately captures the singularity. Starting our main MinRes procedure from this adapted mesh and proceeding with uniform refinement, the estimator's convergence rate successfully realigns with the true error rate, as depicted in \Cref{Fig:Case2_Rates} (center).

To establish a standalone procedure independent of the auxiliary initial MinRes step, we employ the main MinRes method directly with an adaptive mesh refinement strategy driven by our a posteriori error estimator $\eta$, starting from the original coarse mesh. For the adaptive process, we employ Dörfler marking \cite{dorfler_sinum96} with a parameter $\theta = 0.5$. As illustrated in \Cref{Fig:Case2_Rates} (right), the adaptive refinement dynamically resolves the singularity of the right-hand side as the iterative process advances. Consequently, the estimator accurately tracks the actual error rate without requiring a manually pre-adapted mesh. 

\begin{figure}[t!]
    \centering
    \includegraphics[width=0.32\textwidth]{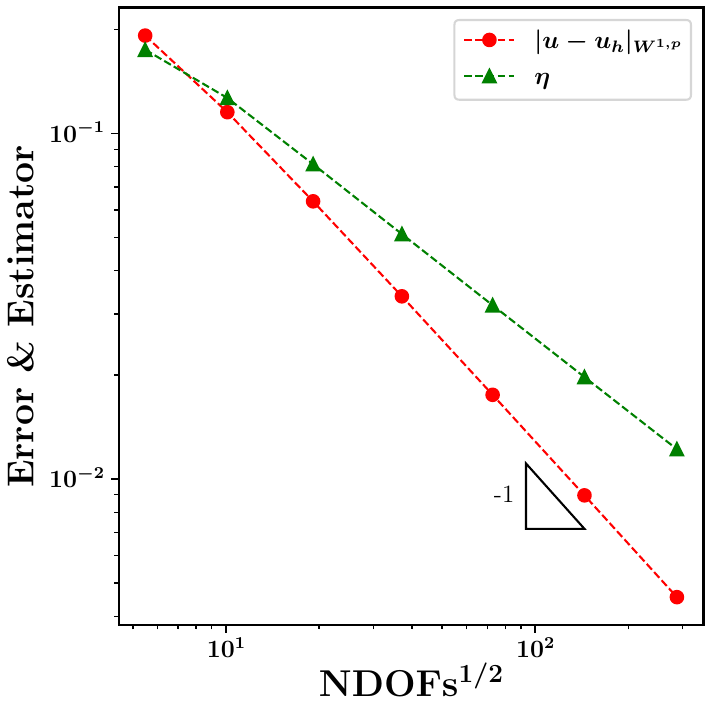}
    \includegraphics[width=0.32\textwidth]{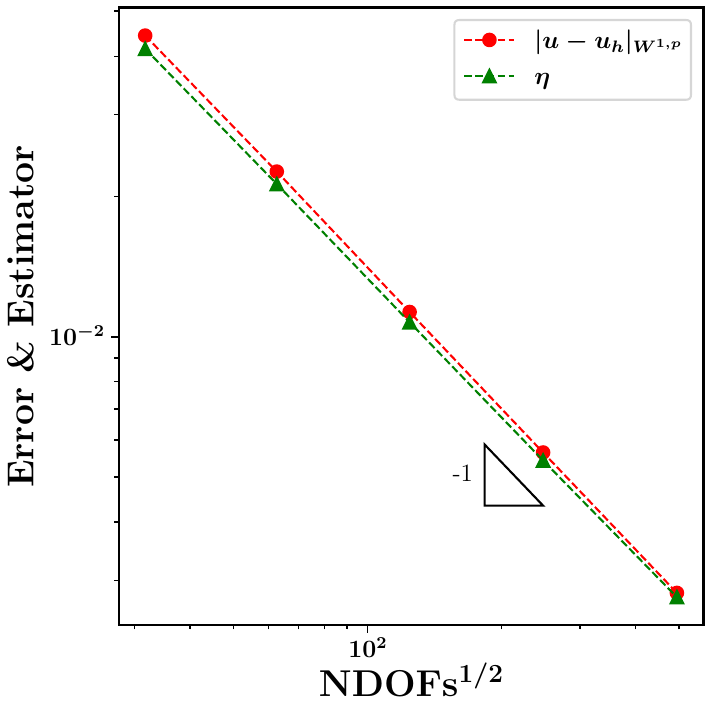}
    \includegraphics[width=0.32\textwidth]{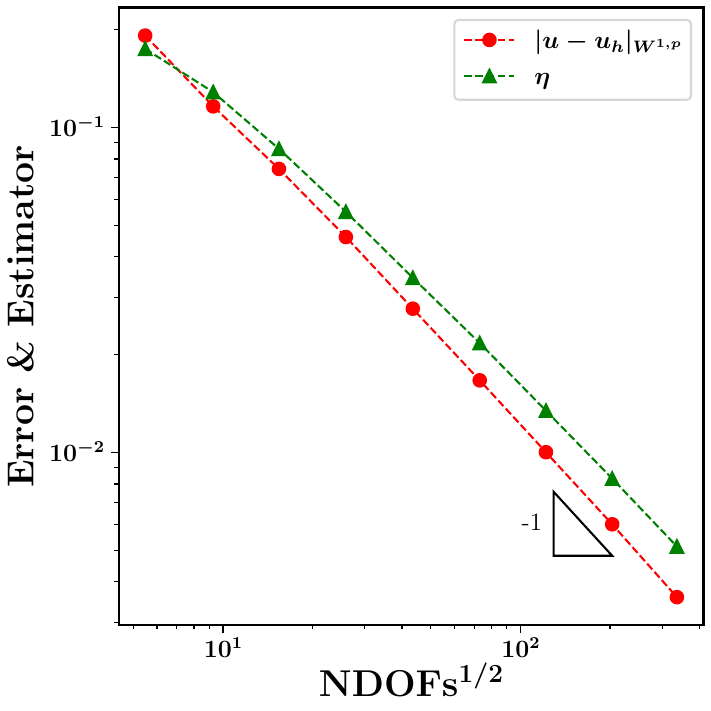}
    \vspace{-2mm}
    \caption{Convergence rates for a singular right-hand side with $p=1.5$ in 2D. Left: Uniform refinement from a standard coarse mesh. Center: Uniform refinement from a pre-adapted initial mesh. Right: Adaptive mesh refinement starting from a standard coarse mesh.}
    \label{Fig:Case2_Rates}
\end{figure}

To visually demonstrate the effectiveness of our a posteriori error estimator in capturing local features, \Cref{Fig:AdaptedMeshes} displays the evolution of the mesh during the adaptive refinement process. The D\"orfler marking strategy strongly localizes the refinement around this singular point, maintaining a coarser resolution in the rest of the domain where the solution is smoother. This targeted adaptivity is precisely what allows the method to recover the previously observed optimal convergence rates.

\begin{figure}[t!]
    \centering
    \includegraphics[width=0.32\textwidth]{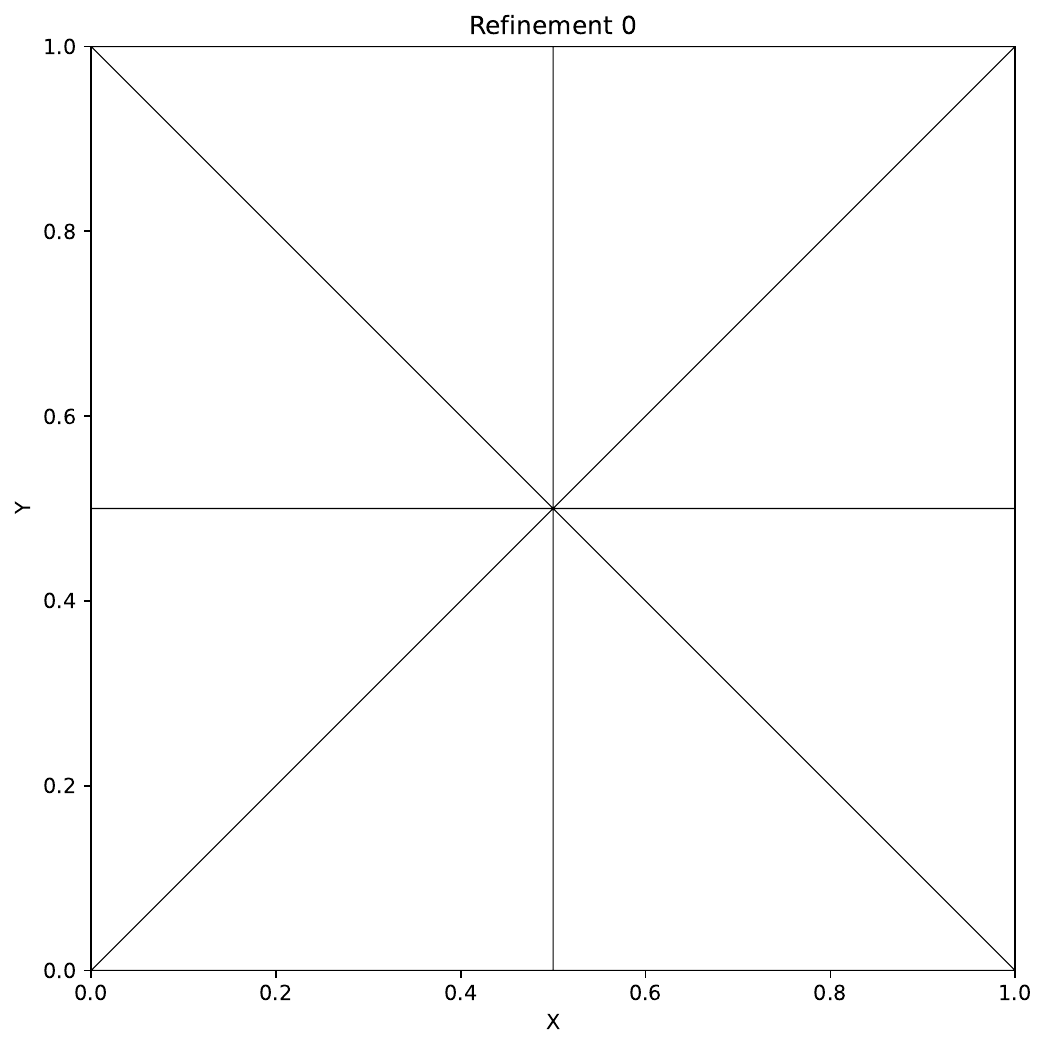}
    \includegraphics[width=0.32\textwidth]{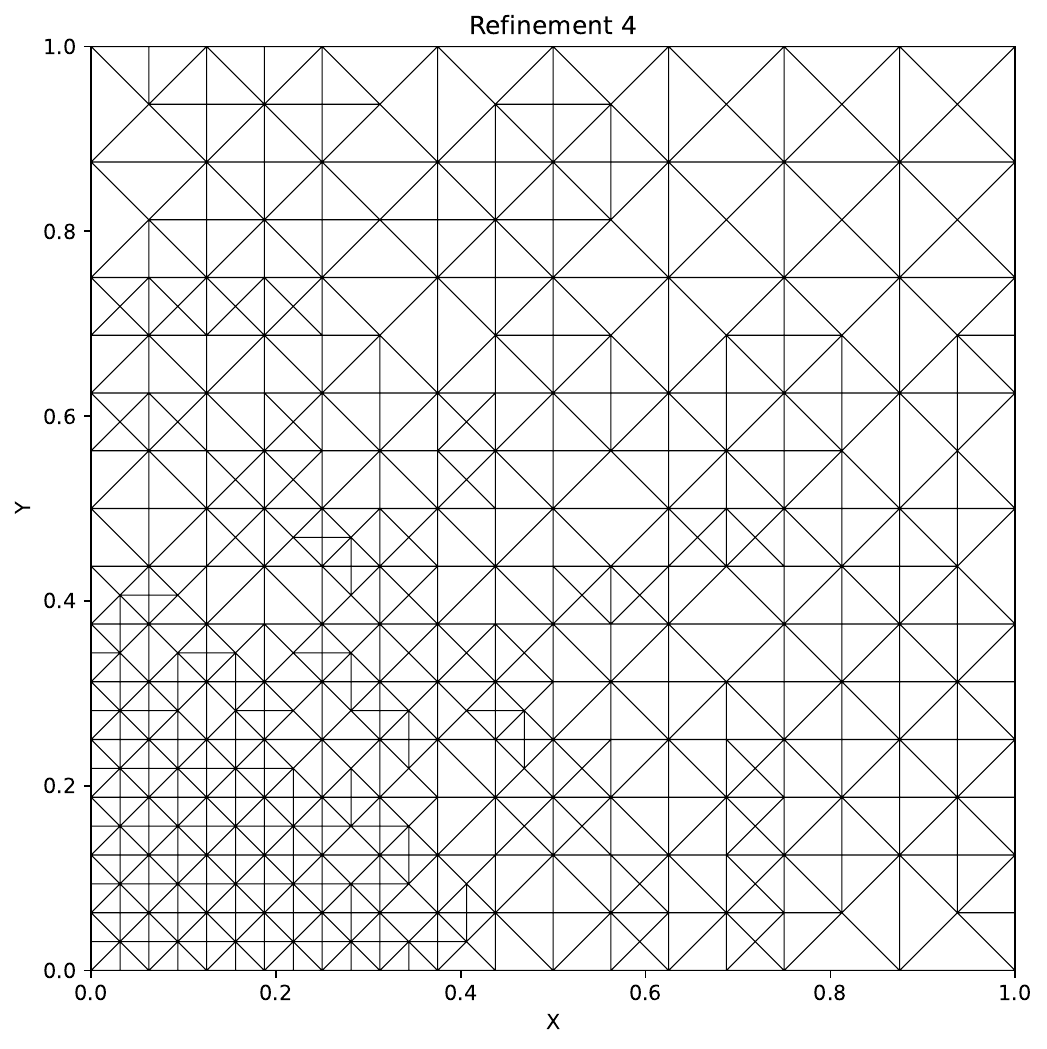}
    \includegraphics[width=0.32\textwidth]{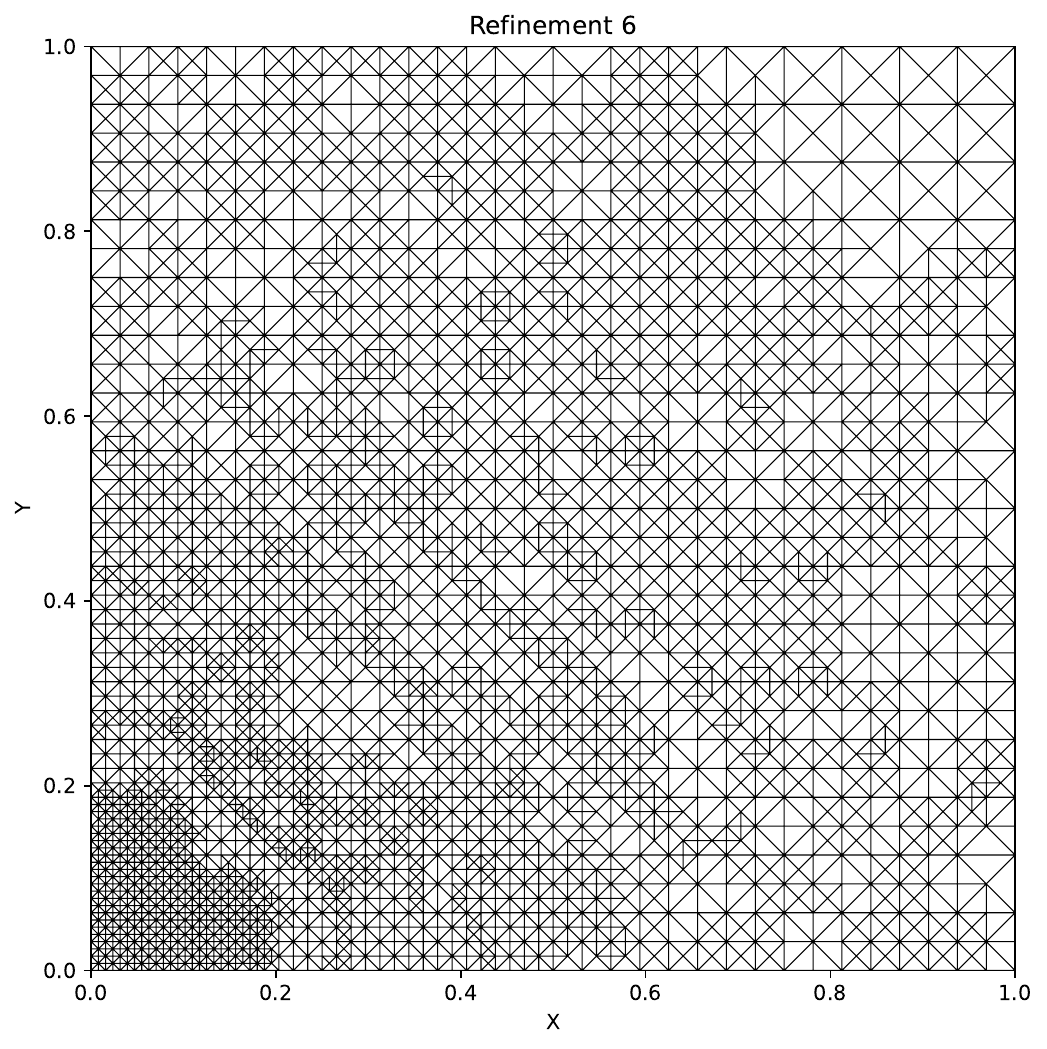}
    \vspace{-2mm}
    \caption{Snapshots of the adaptively refined mesh in 2D for the singular problem ($p=1.5$). From left to right: the initial coarse mesh, an early adapted mesh, and the mesh at the sixth refinement step.}
    \label{Fig:AdaptedMeshes}
\end{figure}

Finally, \Cref{Tab:NewtonIters_Advanced} presents the performance of the nonlinear solver for this singular problem using both the pre-adapted and the fully adaptive mesh strategies. Consistent with the observations in the smooth scenarios, the solver retains its remarkable robustness. The accumulated iteration counts grow very slowly and remain stable, even as the adaptive procedure significantly increases the local resolution and the total number of degrees of freedom.

\begin{table}[t!]
    \centering
    \small
    \begin{tabular}{c c | c c}
        \toprule
       % \multicolumn{4}{c}{{Case 2 (2D singular, $p=1.5$)}} \\
       % \midrule
        \multicolumn{2}{c|}{{Pre-adapted initial mesh}} & \multicolumn{2}{c}{{Adaptive mesh refinement}} \\
        \cmidrule{1-4}
         {NDOFs} &  {Total Newton iterations} &  {NDOFs} & {Total Newton iterations} \\
        \midrule
        1006   & 52  & 30    & 33 \\
        3919   & 80  & 86    & 40 \\
        15460  & 101 & 238   & 50 \\
        61402  & 176 & 671   & 85 \\
        244726 & 160 & 1891  & 88 \\
        -      & -   & 5330  & 83 \\
        -      & -   & 14861 & 92 \\
        -      & -   & 41457 & 134 \\
        -      & -   & 112921 & 176 \\
        \bottomrule
    \end{tabular}
    \caption{Case 2. Total accumulated Newton iterations required to reach the target exponent $p=1.5$ for the singular problem in 2D.}
    \label{Tab:NewtonIters_Advanced}
\end{table}

%%%%%%%%%%%%%%%%%%%%%%%%%%%%%%%%%%%%%%%%%%%%%%%%%%%%%%%%%%%%%%%%%%%%%%%%%%%%%%%%%%%%%%
%%%%%%%%%%%%%%%%%%%%%%%%%%%%%%%%%%%%%%%%%%%%%%%%%%%%%%%%%%%%%%%%%%%%%%%%%%%%%%%%%%%%%%
%%%%%%%%%%%%%%%%%%%%%%%%%%%%%%%%%%%%%%%%%%%%%%%%%%%%%%%%%%%%%%%%%%%%%%%%%%%%%%%%%%%%%%
%%%%%%%%%%%%%%%%%%%%%%%%%%%%%%%%%%%%%%%%%%%%%%%%%%%%%%%%%%%%%%%%%%%%%%%%%%%%%%%%%%%%%%
%%%%%%%%%%%%%%%%%%%%%%%%%%%%%%%%%%%%%%%%%%%%%%%%%%%%%%%%%%%%%%%%%%%%%%%%%%%%%%%%%%%%%%
%%%%%%%%%%%%%%%%%%%%%%%%%%%%%%%%%%%%%%%%%%%%%%%%%%%%%%%%%%%%%%%%%%%%%%%%%%%%%%%%%%%%%%
%%%%%%%%%%%%%%%%%%%%%%%%%%%%%%%%%%%%%%%%%%%%%%%%%%%%%%%%%%%%%%%%%%%%%%%%%%%%%%%%%%%%%%
%%%%%%%%%%%%%%%%%%%%%%%%%%%%%%%%%%%%%%%%%%%%%%%%%%%%%%%%%%%%%%%%%%%%%%%%%%%%%%%%%%%%%%
%%%%%%%%%%%%%%%%%%%%%%%%%%%%%%%%%%%%%%%%%%%%%%%%%%%%%%%%%%%%%%%%%%%%%%%%%%%%%%%%%%%%%%
%%%%%%%%%%%%%%%%%%%%%%%%%%%%%%%%%%%%%%%%%%%%%%%%%%%%%%%%%%%%%%%%%%%%%%%%%%%%%%%%%%%%%%
%%%%%%%%%%%%%%%%%%%%%%%%%%%%%%%%%%%%%%%%%%%%%%%%%%%%%%%%%%%%%%%%%%%%%%%%%%%%%%%%%%%%%%
%%%%%%%%%%%%%%%%%%%%%%%%%%%%%%%%%%%%%%%%%%%%%%%%%%%%%%%%%%%%%%%%%%%%%%%%%%%%%%%%%%%%%%
%%%%%%%%%%%%%%%%%%%%%%%%%%%%%%%%%%%%%%%%%%%%%%%%%%%%%%%%%%%%%%%%%%%%%%%%%%%%%%%%%%%%%%
%%%%%%%%%%%%%%%%%%%%%%%%%%%%%%%%%%%%%%%%%%%%%%%%%%%%%%%%%%%%%%%%%%%%%%%%%%%%%%%%%%%%%%

\section{Concluding remarks}\label{sec:Conclusions}

In this work, we develop and analyze a residual minimization method for solving nonlinear PDEs in Banach spaces, with a particular focus on the $p$-Laplacian problem. By formulating the problem as a minimization of the residual in a dual norm, we derived a mixed formulation that involves the duality map of the test space. This approach naturally extends the framework proposed in \cite{Muga2020Discretization} to nonlinear operators.

From a theoretical standpoint, we established the well-posedness of both the continuous and discrete minimization problems, relying on the strict convexity and reflexivity of the underlying Banach spaces. A key contribution of this study is the derivation of an a posteriori error estimator. We proved that the norm of the residual representative provides both upper and lower bounds for the approximation error, thereby guaranteeing reliability and efficiency. In the discrete setting, using a non-conforming Discontinuous Petrov--Galerkin (DPG) framework with broken Sobolev norms, we showed that the method remains stable and convergent.

The numerical experiments validated our theoretical findings. For smooth solutions, the method exhibits optimal convergence rates for both $p < 2$ and $p \geq 2$. In the presence of singularities, the proposed a posteriori error estimator drives an adaptive mesh refinement strategy. This adaptivity enables the method to recover optimal algebraic convergence rates, significantly outperforming uniform refinement, particularly for degenerate diffusion coefficients ($p=3$).

Apart from the evident benefits of straightforwardly designing continuous a posteriori error estimators, we believe that the mechanisms developed in this work can be used to construct operator preconditioners for linear and nonlinear problems in $L^p$ spaces (see, e.g., \cite{das26robust}).  Moreover, these ideas may be useful for MinRes-based neural network approximation of nonlinear PDEs, in the spirit of \cite{fuhrer2025posteriori}.

%%%%%%%%%%%%%%%%%%%%%%%%%%%%%%%%%%%%%%%%%%%%%%%%%%%%%%%%%%%%%%%%%%%%%%%%%%%%%%%%%%%%%%
%%%%%%%%%%%%%%%%%%%%%%%%%%%%%%%%%%%%%%%%%%%%%%%%%%%%%%%%%%%%%%%%%%%%%%%%%%%%%%%%%%%%%%
%%%%%%%%%%%%%%%%%%%%%%%%%%%%%%%%%%%%%%%%%%%%%%%%%%%%%%%%%%%%%%%%%%%%%%%%%%%%%%%%%%%%%%
%%%%%%%%%%%%%%%%%%%%%%%%%%%%%%%%%%%%%%%%%%%%%%%%%%%%%%%%%%%%%%%%%%%%%%%%%%%%%%%%%%%%%%
%%%%%%%%%%%%%%%%%%%%%%%%%%%%%%%%%%%%%%%%%%%%%%%%%%%%%%%%%%%%%%%%%%%%%%%%%%%%%%%%%%%%%%
%%%%%%%%%%%%%%%%%%%%%%%%%%%%%%%%%%%%%%%%%%%%%%%%%%%%%%%%%%%%%%%%%%%%%%%%%%%%%%%%%%%%%%
%%%%%%%%%%%%%%%%%%%%%%%%%%%%%%%%%%%%%%%%%%%%%%%%%%%%%%%%%%%%%%%%%%%%%%%%%%%%%%%%%%%%%%
%%%%%%%%%%%%%%%%%%%%%%%%%%%%%%%%%%%%%%%%%%%%%%%%%%%%%%%%%%%%%%%%%%%%%%%%%%%%%%%%%%%%%%
%%%%%%%%%%%%%%%%%%%%%%%%%%%%%%%%%%%%%%%%%%%%%%%%%%%%%%%%%%%%%%%%%%%%%%%%%%%%%%%%%%%%%%
%%%%%%%%%%%%%%%%%%%%%%%%%%%%%%%%%%%%%%%%%%%%%%%%%%%%%%%%%%%%%%%%%%%%%%%%%%%%%%%%%%%%%%
%%%%%%%%%%%%%%%%%%%%%%%%%%%%%%%%%%%%%%%%%%%%%%%%%%%%%%%%%%%%%%%%%%%%%%%%%%%%%%%%%%%%%%
%%%%%%%%%%%%%%%%%%%%%%%%%%%%%%%%%%%%%%%%%%%%%%%%%%%%%%%%%%%%%%%%%%%%%%%%%%%%%%%%%%%%%%
%%%%%%%%%%%%%%%%%%%%%%%%%%%%%%%%%%%%%%%%%%%%%%%%%%%%%%%%%%%%%%%%%%%%%%%%%%%%%%%%%%%%%%
%%%%%%%%%%%%%%%%%%%%%%%%%%%%%%%%%%%%%%%%%%%%%%%%%%%%%%%%%%%%%%%%%%%%%%%%%%%%%%%%%%%%%%

\bibliographystyle{siam}  % O usa un estilo como plain, alpha, etc.
\bibliography{References}  % Asegúrate de que coincida el nombre con tu archivo .bib

@book{Bochev2009Least,
  title={Least-squares finite element methods},
  author={Bochev, P. B. and Gunzburger, M.},
  volume={166},
  year={2009},
  publisher={Springer Science \& Business Media}
}

@book{Brenner2008Mathematical,
  title={The mathematical theory of finite element methods},
  author={Brenner, Susanne C and Scott, L Ridgway},
  year={2008},
  publisher={Springer}
}

@book{Ciarlet2013Linear,
   title =     {Linear and Nonlinear Functional Analysis with Applications},
   author =    {Ciarlet, Philippe G},
   publisher = {SIAM-Society for Industrial and Applied Mathematics},
   isbn =      {1611972582; 9781611972580},
   year =      {2013},
   url =       {libgen.li/file.php?md5=44705a4590966cd2d7bc97073033b379}}

@article{Ciarlet2018Family,
  title={A family of Crouzeix--Raviart finite elements in 3D},
  author={Ciarlet Jr, Patrick and Dunkl, Charles F and Sauter, Stefan A},
  journal={Analysis and Applications},
  volume={16},
  number={05},
  pages={649--691},
  year={2018},
  publisher={World Scientific}
}

@book{Cioranescu1990Geometry,
   title =     {Geometry of Banach Spaces, Duality Mappings and Nonlinear Problems},
   author =    {Cioranescu, I.},
   publisher = {Springer},
   isbn =      {9789401074544; 9401074542; 9789400921214; 9400921217},
   year =      {1990},
   series =    {Mathematics and Its Applications : Main Series    №62},
   edition =   {1},
url =       {libgen.li/file.php?md5=e331fb8331507f7ea760a8e2fd641ebb}
}

@article{Crouzeix1973Conforming,
  title={Conforming and nonconforming finite element methods for solving the stationary {S}tokes equations {I}},
  author={Crouzeix, Michel and Raviart, P-A},
  journal={Revue Fran{\c{c}}aise d'Automatique Informatique Recherche Op{\'e}rationnelle. Math{\'e}matique},
  volume={7},
  number={R3},
  pages={33--75},
  year={1973},
  publisher={EDP Sciences}
}

@book{Deimling2013Nonlinear,
  title={Nonlinear functional analysis},
  author={Deimling, Klaus},
  year={2013},
  publisher={Springer Science \& Business Media}
}

@book{Lindqvist2017Notes,
  title={Notes on the p--{L}aplace equation},
  author={Lindqvist, P.},
  number={161},
  year={2017},
  publisher={University of Jyv{\"a}skyl{\"a}}
}

@article{fuhrer2025posteriori,
  title={A posteriori analysis of neural network approximations},
  author={F{\"u}hrer, Thomas and Rojas, Sergio},
  journal={arXiv preprint arXiv:2507.06017},
  year={2025}
}

@article{demkowicz2025discontinuous,
  title={{The discontinuous Petrov--Galerkin method}},
  author={Demkowicz, Leszek and Gopalakrishnan, Jay},
  journal={Acta Numerica},
  volume={34},
  pages={293--384},
  year={2025},
  publisher={Cambridge University Press}
}

@article{Balci2023Relaxed,
  title={Relaxed {K}a{\v{c}}anov Scheme for the $p$-{L}aplacian with Large Exponent},
  author={Balci, Anna Kh and Diening, Lars and Storn, Johannes},
  journal={SIAM Journal on Numerical Analysis},
  volume={61},
  number={6},
  pages={2775--2794},
  year={2023},
  publisher={SIAM}
}

@article{Cantin2018A,
  title={A {DPG} framework for strongly monotone operators},
  author={Cantin, Pierre and Heuer, Norbert},
  journal={SIAM Journal on Numerical Analysis},
  volume={56},
  number={5},
  pages={2731--2750},
  year={2018},
  publisher={SIAM}
}

@article{Carstensen2018Nonlinear,
  title={Nonlinear {D}iscontinuous {P}etrov--{G}alerkin methods},
  author={Carstensen, C. and Bringmann, P. and Hellwig, F. and Wriggers, P.},
  journal={Numerische Mathematik},
  volume={139},
  number={3},
  pages={529--561},
  year={2018},
  id={Carstensen2018},
  isbn={0945-3245},
  doi={https://doi.org/10.1007/s00211-018-0947-5},
  publisher={Springer}
}

@article{Cier2021Automatically,
title = {Automatically adaptive stabilized finite elements and continuation analysis for compaction banding in geomaterials},
author = {Cier, R. J. and Poulet, T. and Rojas, S. and Veveakis, M. and Calo, V. M.},
journal = {International Journal for Numerical Methods in Engineering},
volume = {122},
number = {21},
pages = {6234-6252},
doi = {https://doi.org/10.1002/nme.6790},
url = {https://onlinelibrary.wiley.com/doi/abs/10.1002/nme.6790},
eprint = {https://onlinelibrary.wiley.com/doi/pdf/10.1002/nme.6790},
year = {2021}
}

@article{Cier2021ANonlinear,
title = {A nonlinear weak constraint enforcement method for advection-dominated diffusion problems},
author = {Cier, R. J. and Rojas, S. and Calo, V. M.},
journal = {Mechanics Research Communications},
volume = {112},
pages = {103602},
year = {2021},
note = {Special issue honoring G.I. Taylor Medalist Prof. Arif Masud},
issn = {0093-6413},
doi = {https://doi.org/10.1016/j.mechrescom.2020.103602},
url = {https://www.sciencedirect.com/science/article/pii/S0093641320301300}
}

@article{das26robust,
 author = {Das, Rishi and Hutridurga, Harsha and Pani, Amiya K. and Ruiz-Baier, Ricardo},
 doi = {10.48550/arXiv.2510.24527},
 journal = {SIAM Journal on Numerical Analysis},
 title = {Robust stability and preconditioning of {Darcy--Forchheimer} equations},
 year = {2026},
 pages = {under review}
}

@article{dorfler_sinum96,
author = {D\"orfler, Willy},
title = {A Convergent Adaptive Algorithm for {P}oisson's Equation},
journal = {SIAM {J}ournal on {N}umerical {A}nalysis}, 
volume = {33},
number = {3},
pages = {1106--1124},
year = {1996}}

@article{glowinski1975sur,
  title={Sur l'approximation, par {\'e}l{\'e}ments finis d'ordre un, et la r{\'e}solution, par p{\'e}nalisation-dualit{\'e} d'une classe de probl{\`e}mes de Dirichlet non lin{\'e}aires},
  author={Glowinski, Roland and Marroco, Americo},
  journal={Revue fran{\c{c}}aise d'automatique, informatique, recherche op{\'e}rationnelle. Analyse num{\'e}rique},
  volume={9},
  number={R2},
  pages={41--76},
  year={1975},
  publisher={EDP Sciences}
}

@article{Li2022An,
  title={An {$L^p$-DPG} Method with Application to {2D} Convection-Diffusion Problems},
  author={Li, Jiaqi and Demkowicz, Leszek},
  journal={Computational Methods in Applied Mathematics},
  volume={22},
  number={3},
  pages={649--662},
  year={2022},
  publisher={De Gruyter}
}

@article{Liu1994Quasi,
  title={Quasi-norm error bounds for the finite element approximation of some degenerate quasilinear elliptic equations and variational inequalities},
  author={Liu, W. and Barrett, J. W.},
  journal={ESAIM: Mathematical Modelling and Numerical Analysis},
  volume={28},
  number={6},
  pages={725--744},
  year={1994},
  publisher={EDP Sciences}
}

@article{Liu2001Some,
  title={Some a posteriori error estimators for p--{L}aplacian based on residual estimation or gradient recovery},
  author={Liu, W. and Yan, N.},
  journal={Journal of Scientific Computing},
  volume={16},
  pages={435--477},
  year={2001},
  publisher={Springer}
}

@article{Millar2022Projection,
  title={Projection in negative norms and the regularization of rough linear functionals},
  author={Millar, F. and Muga, I. and Rojas, S. and Van der Zee, K. G.},
  journal={Numerische Mathematik},
  volume={150},
  number={4},
  pages={1087--1121},
  year={2022},
  publisher={Springer}
}

@article{Monsuur2024Minimal,
  title={Minimal residual methods in negative or fractional {S}obolev norms},
  author={Monsuur, Harald and Stevenson, Rob and Storn, Johannes},
  journal={Mathematics of Computation},
  volume={93},
  number={347},
  pages={1027--1052},
  year={2024}
}

@article{Muga2020Discretization,
  title={Discretization of Linear Problems in {B}anach Spaces: Residual Minimization, Nonlinear {P}etrov--{G}alerkin, and Monotone Mixed Methods},
  author={Muga, I. and Van Der Zee, K. G.},
  journal={SIAM Journal on Numerical Analysis},
  volume={58},
  number={6},
  pages={3406--3426},
  year={2020}
}

@article{Munoz2023Multistage,
  title={Multistage {DPG} time-marching scheme for nonlinear problems},
  author={Mu{\~n}oz-Matute, Judit and Demkowicz, Leszek},
  journal={arXiv preprint arXiv:2309.00069},
  year={2023}
}

@article{Rojas2024Robust,
  title={Robust {V}ariational {P}hysics-{I}nformed {N}eural {N}etworks},
  author={Rojas, S. and Maczuga, P. and Mu{\~n}oz-Matute, J. and Pardo, D. and Paszy{\'n}ski, M.},
  journal={Computer Methods in Applied Mechanics and Engineering},
  volume={425},
  pages={116904},
  year={2024},
  publisher={Elsevier}
}

@article{Stern2015Banach,
  author  = {Stern, A.},
  title   = {{Banach} space projections and {Petrov}-{Galerkin} estimates},
  journal = {Numerische Mathematik},
  volume  = {130},
  year    = {2015},
  pages   = {125--133},
}

@article{Storn2024Solving,
  title={Solving Minimal Residual Methods in {$W^{-1, p'}$} with Large Exponents $p$},
  author={Storn, Johannes},
  journal={Journal of Scientific Computing},
  volume={99},
  number={2},
  pages={35},
  year={2024},
  publisher={Springer}
}

\end{document}